\documentclass[11pt]{amsart}

\usepackage[utf8]{inputenc}
\usepackage[backend=biber,style=alphabetic,maxbibnames=50]{biblatex}
\usepackage{amssymb,amsfonts}
\usepackage{amsmath,amscd}
\usepackage[all,arc]{xy}
\usepackage{enumerate}
\usepackage{mathrsfs}
\usepackage[pdftex]{graphicx}
\usepackage[T1]{fontenc}
\usepackage[usenames,dvipsnames]{color}
\usepackage[bookmarks=false]{hyperref}
\usepackage{dsfont}
\usepackage{tikz-cd}
\usetikzlibrary{automata,positioning}
\usepackage{fullpage}

\usepackage{color}
\usepackage{xcolor}

\theoremstyle{plain}
\newtheorem{thm}{Theorem}[section]
\newtheorem{cor}[thm]{Corollary}
\newtheorem{prop}[thm]{Proposition}
\newtheorem{lem}[thm]{Lemma}

\theoremstyle{definition}
\newtheorem{defn}[thm]{Definition}

\newtheorem{example}[thm]{Example}

\newtheorem{aDD^+m}[thm]{ADD^+endum}

\theoremstyle{remark}
\newtheorem{rmk}[thm]{Remark}

\newcommand{\bbC}{\mathbb{C}} 

\newcommand{\bbR}{\mathbb{R}} 

\newcommand*{\defeq}{\mathrel{\vcenter{\baselineskip0.5ex \lineskiplimit0pt
			\hbox{\scriptsize.}\hbox{\scriptsize.}}}%
	=}

\title{Analyticity of the Hausdorff dimension and 
metric structures on Misiurewicz families of polynomials}
\author{Fabrizio Bianchi}
\address{Dipartimento di Matematica, Università di Pisa, Largo Bruno Pontecorvo 5, 56127 Pisa, Italy}
 \email{fabrizio.bianchi$@$unipi.it}
\author{Yan Mary He}
\address{Department of Mathematics\\
	University of Oklahoma\\
	Norman, OK 73019}
\email{he$@$ou.edu}
\date{\today}

\begin{document}

\begin{abstract}
Consider a holomorphic
family
$(f_\lambda)_{\lambda \in \Lambda}$
of polynomial
maps on $\mathbb C$
with the property that a critical point of $f_\lambda$ is persistently preperiodic to a repelling periodic point of $f_\lambda$.
Let $\Omega$ be a bounded stable component of $\Lambda$ with the property that, for all $\lambda \in \Omega$, all the other critical points of $f_\lambda$
belong to attracting basins. 
In this paper, we introduce a dynamically meaningful 
geometry on $\Omega$ by constructing a natural path metric on $\Omega$ coming from a
2-form $\langle \cdot, \cdot \rangle_G$. Our construction uses thermodynamic formalism. A key ingredient is the
spectral gap of adapted transfer operators on suitable Banach spaces, which also implies the analyticity of $\langle \cdot, \cdot \rangle_G$ on the unit tangent bundle of $\Omega$.
As part of our construction, we recover
a result of Skorulski and
Urba{\'n}ski stating
that the Hausdorff dimension of the Julia set of $f_\lambda$
varies analytically over $\Omega$.
\end{abstract}

\maketitle

\section{Introduction}
Let $S$ be a closed surface of genus 
$g \geq 2$.
The Teichm\"uller space $T(S)$ of $S$, which parametrizes the hyperbolic structures on $S$, plays a fundamental role in modern mathematics.
The topology and geometry of
Teichm\"uller spaces
have been investigated from numerous viewpoints. Ahlfors 
\cite{Ahlfors_QCMappings}
proved that 
$T(S)$ is homeomorphic to $\mathbb R^{6g-6}$.
On the other hand, $T(S)$
carries a number of natural metrics defined from different perspectives, e.g., the Teichm\"uller metric, the Weil-Petersson metric, and the Thurston metric; see
for instance \cite{Hubbard06,ImaTan}. 

From the perspective of Sullivan's dictionary \cite{Sullivan85, McMullen94, DSU17},
stable components
(in the sense of \cite{Lyu83typical,MSS83})
of moduli spaces of 
rational maps (seen as holomorphic dynamical systems on the Riemann sphere) are the natural counterparts in complex dynamics of Teichm\"uller spaces of closed surfaces. This correspondence provides both the motivation and the tools to examine
both the topology and the geometry 
of such stable components, in parallel to the theories of Teichm\"uller spaces.

Ever since McMullen's \cite{McMullen08} 
construction of the Weil-Petersson metric on the space of degree $d \ge 2$ Blaschke products in complex dynamics, there has been an extensive study of Weil-Petersson metrics on stable components of moduli spaces. Ivrii \cite{Ivrii14} studied the completeness properties of McMullen's metric for degree 2 Blaschke products. 
Nie and the second author
\cite{HeNie23}
constructed Weil-Petersson metrics on general hyperbolic components in moduli spaces of rational maps. Lee, Park, and the second author
\cite{HLP23,HLP25}
studied the degeneracy
loci of the Weil-Petersson metric on spaces of quasi-Blaschke products. 
In our earlier paper \cite{BH24}, we 
studied the Weil-Petersson metric on stable components of 
polynomials families with a persistent parabolic point.

In this paper, we extend the theory of Weil-Petersson metrics in complex dynamics to stable components of a {\it Misiurewicz family} of polynomials (i.e., a family where some critical point is persistently preperiodic to some repelling periodic point).
The construction of the Weil-Petersson metric
requires a deep analysis of the
spectral properties of adapted transfer operators to deal with the presence of critical points in the Julia sets and the lack of uniform hyperbolicity.

\subsection{Statement of results}
Denote by ${\rm Poly}^{cm}_D$ (resp.\ ${\rm Rat}^{cm}_D$) the space of critically marked degree $D \ge 2$ polynomials (resp.\ rational maps). 
A not necessarily closed
subfamily $\Lambda$ of ${\rm Poly}^{cm}_D$ (resp.\ ${\rm Rat}^{cm}_D$) is a {\it Misiurewicz family} if it is the open subset of a 
family given by a finite number of 
critical relations of the form
\[
f^{n_i} (c_{i} (\lambda)) = f^{n_i+m_i} (c_{i}(\lambda)) = r_{i}(\lambda),
\]
where $c_i(\lambda)$ is the $i$-th marked critical point of $f_\lambda$, $n_i,m_i$ are positive integers,
and
each $r_{i}(\lambda)$ is a repelling periodic point for every $f_\lambda\in \Lambda$.
The dynamics of a Misiurewicz polynomial $f_\lambda \in \Lambda$ on its Julia set is not uniformly hyperbolic, due to the presence of critical points.
We say that a stable component $\Omega\Subset \Lambda$ is {\it $\Lambda$-hyperbolic} if,
for every $\lambda \in \Omega$, every critical point $c_j(\lambda)$ such that $c_j$ is active on $\Lambda$ is contained in the basin of some attracting cycle for $f_\lambda$.

Our main goal in this paper
is to define a
Weil-Petersson
path-metric
on a bounded $\Lambda$-hyperbolic component $\Omega$ of a Misiurewicz subfamily $\Lambda$ of ${\rm poly}_D^{cm}$. To this end, we construct a positive semi-definite symmetric bilinear form $\langle \cdot, \cdot \rangle_G$ on each tangent space $T_{\lambda}\Omega$, which will turn out to be equivalent to the pressure 2-form.
The construction of $\langle \cdot, \cdot \rangle_G$ is valid on any $\Lambda$-hyperbolic component of a Misiurewicz subfamily $\Lambda$ of ${\rm rat}_D^{cm}$.

As the first step of our construction of the 2-form $\langle \cdot, \cdot \rangle_G$, we show that the Hausdorff dimension function is real-analytic.

\begin{thm}\label{thm_ana_main}
Let $\Omega$ be a $\Lambda$-hyperbolic component of a Misiurewicz subfamily $\Lambda$ of ${\rm rat}_D^{cm}$. The Hausdorff dimension function $\delta \colon \Omega \to (0,2)$ sending $\lambda$ to the Hausdorff dimension of the Julia set of $f_\lambda$ is real-analytic.
\end{thm}

In fact, Theorem \ref{thm_ana_main} is obtained as a corollary (see Corollary \ref{c:delta-ana}) of a stronger analyticity result;
see Theorem \ref{t:new-analytic-all}. We also remark that Theorem \ref{thm_ana_main} is already known; see for example Skorulski-Urbanski \cite{SkorulskiUrbanski14}. Our proof of Theorem \ref{t:new-analytic-all}, and hence of Theorem \ref{thm_ana_main},
is very different from Skorulski-Urbanski's proof \cite{SkorulskiUrbanski14}. The framework of our
proof will be crucial in later parts of the paper. More details of the proof strategies will be given in Section \ref{sec_proof_strategy}.

\medskip

We then construct the Weil-Petersson metric
on each tangent space $T_{\lambda}\Omega$. The construction follows the general framework of \cite{HeNie23, BH24}. 

\medskip

By construction, the 2-form $\langle \cdot, \cdot \rangle_G$ may not be non-degenerate; namely, there may exist a non-zero tangent vector $\vec{v}\in T_\lambda\Omega$ such that $\langle \vec{v}, \vec{v} \rangle_G = 0$. For example, in \cite{HLP23,HLP25}, Lee, Park and the second author studied the degeneracy
loci of the Weil-Petersson metric on spaces of quasi-Blaschke products. On the other hand, in \cite{HeNie23}, Nie and the second author gave a sharp condition under which the Weil-Petersson metric is non-degenerate on (uniformly) hyperbolic components of rational maps.
This result deeply exploited the uniform hyperbolicity
and in particular a result by Oh-Winter  \cite{Oh17}
on the distribution of the
multiplier of the periodic points.
Even though we cannot rely on that result in our Misiurewicz setting, we can still prove that $\langle \cdot, \cdot \rangle_G$ defines a path-metric on $\Omega$.

\begin{thm}\label{thm_main}
Let $\Omega$ be a bounded $\Lambda$-hyperbolic component of a Misiurewicz subfamily $\Lambda$ of ${\rm poly}_D^{cm}$.
Then the function $d_G: \Omega \times \Omega \to \mathbb R$ given by
\begin{equation*}
d_G(x,y) \defeq \inf_\gamma \int_0^1 \|\gamma'(t)\|_G dt
\end{equation*}
is a distance function. Here
the infimum is taken over all the $C^1$-paths $\gamma$ connecting $x$ to $y$ in $\Omega$.
\end{thm}

The above theorem parallels the result we obtained 
in \cite{BH24} for the case of parabolic families of polynomials. We note that the analysis in \cite{BH24} does not apply to the case of Misiurewicz maps. More specifically, for parabolic maps, the geometric potential is H\"older continuous, hence we could apply  
the machinery of \cite{BD23eq1, BD24eq2} to the transfer orerator associated to such weights and their perturbations.
For Misiurewicz maps, we will instead
consider suitable transfer operators on an enlarged tower space, where the dynamics becomes expanding.

\subsection{Strategies of the proofs} \label{sec_proof_strategy}
The main technical difficulty of the paper lies in the proof of Theorem \ref{t:new-analytic-all},
which states that 
a
certain pressure function $(t_1,\lambda_1, t_2, \lambda_2) \mapsto P(t_1,\lambda_1, t_2, \lambda_2)$ related to two geometric potentials
is real-analytic. As consequences of Theorem \ref{t:new-analytic-all}, we obtain Theorem \ref{thm_ana_main} as well as the analyticity of the two-form $\langle \cdot, \cdot \rangle_G$ on the unit tangent bundle of $\Omega$. 

To prove Theorem \ref{t:new-analytic-all}, we first consider
a suitable extension
$T \colon \mathcal{T} \to \mathcal{T}$
of the dynamics to an enlarged \emph{tower space}, 
see Section \ref{sec_3.2_tower},
which provides a way to deal with critical points in the Julia set for a single Misiurewicz map. Since we will work with a stable component of Misiurewicz maps, we generalize the work
of Makarov-Smirnov \cite{MakSmir03} to a holomorphic family of tower maps $\{T_\lambda \colon \mathcal{T}_\lambda \to \mathcal{T}_\lambda\}_{\lambda\in \Omega}$. A key observation is that the H\"older continuous conjugation between Julia sets in the same 
$\Lambda$-hyperbolic 
component extends to a H\"older continous conjugation between the corresponding tower dynamics; see Lemma \ref{l:Hlambda-logq}.

For $t_i \in \mathbb R, \lambda_i \in \Omega$, we construct a 4-parameter family of transfer operators $\mathcal{L}_{t_1,\lambda_1,t_2,\lambda_2} \colon C^\kappa(\mathcal{T}) \to C^\kappa(\mathcal{T})$ acting on the Banach space of $\kappa$-H\"older continuous functions of $\mathcal{T}$. The four parameters of the transfer operators come from the potential functions; namely, we consider a 4-parameter family of potentials $\zeta_{(t_1,\lambda_1,t_2,\lambda_2)}$ to define $\mathcal{L}_{t_1,\lambda_1,t_2,\lambda_2} = \mathcal{L}_{\zeta_{(t_1,\lambda_1,t_2,\lambda_2)}}$. In fact, our potential $\zeta_{(t_1,\lambda_1,t_2,\lambda_2)}$ is of the form $\zeta_{(t_1,\lambda_1)} + \zeta_{(t_2,\lambda_2)}$, where each $\zeta_{(t_i,\lambda_i)}$ is a modified version of the geometric potential $-t_i\log|f_{\lambda_i}'|$.

Then an important step is to prove that such transfer operators have a {\it spectral gap} for certain ranges of the parameters; see Proposition \ref{p:MS-gen}. In particular, the spectral gap property follows from a Lasota-Yorke type of estimate,
which is locally uniform in all the parameters $t_1, \lambda_1, t_2, \lambda_2$.

To show that the pressure function is indeed analytic, and not only separately analytic in its variables, we will employ an extension argument, see for instance \cite{SU10,UZ04real}.
More specifically, we show that we can extend the 4-parameter family of potential functions  $\zeta_{t_1,\lambda_1,t_2,\lambda_2}$ to complex parameters $t_1,\lambda_1,t_2,\lambda_2$, in such a way that the 4-parameter family of transfer operators $\mathcal{L}_{t_1,\lambda_1,t_2,\lambda_2}$ becomes a {\it holomorphic} family of operators; see Proposition \ref{p:SU-full} and Lemma \ref{lem_314}. Then it follows from the perturbation theory that the eigenvalue, and therefore the pressure, varies real-analytically with the parameters.

\subsection{Organization of the paper}
The paper is organized as follows. Section \ref{sec_2} collects basic facts about Misiurewicz maps and $\Lambda$-hyperbolic components. We prove Theorem \ref{t:new-analytic-all} and therefore Theorem \ref{thm_ana_main} in Section \ref{sec_3}. We construct the 2-form $\langle \cdot, \cdot \rangle_G$ and the pressure form in Section \ref{sec_4}. Finally, we prove Theorem \ref{thm_main} in Section \ref{sec_5}.

\subsection*{Acknowledgements}
The authors would like to thank the Institute for Advanced Study, Princeton,
for the support and  hospitality
through the Summer Collaborators Program.
This project has received funding from
 the
 Programme
 Investissement d'Avenir
(ANR QuaSiDy /ANR-21-CE40-0016,
ANR PADAWAN /ANR-21-CE40-0012-01,
ANR TIGerS /ANR-24-CE40-3604),
 from the MIUR Excellence Department Project awarded to the Department of Mathematics of the University of Pisa, CUP I57G22000700001,
 and
 from the PHC Galileo project G24-123.
The first
author is affiliated to the GNSAGA group of INdAM.

\section{Misiurewicz families and $\Lambda$-hyperbolic components}\label{sec_2}
Let $D\ge 2$ be an integer. Let ${\rm Rat}^{cm}_D$ (resp.\ ${\rm Poly}^{cm}_D$) be the space of degree $D$ rational maps (resp.\ polynomials) with marked critical points $c_1,
\ldots, c_{2D-2}$ (resp.\ $c_1, \ldots, c_{D-1}$) and ${\rm rat}^{cm}_D$ (resp.\ ${\rm poly}^{cm}_D$)
the moduli space obtained from ${\rm Rat}^{cm}_D$ (resp.\
${\rm Poly}^{cm}_D$) by taking the quotient by all M\"obius 
(resp.\ affine)
conjugacies. For every 
$\lambda \in {\rm Rat}^{cm}_D$ (resp.\ ${\rm Poly}^{cm}_D$), we denote by $f_\lambda$ the corresponding 
rational map (resp.\ polynomial),
by $\mu_\lambda$ its unique measure of maximal entropy \cite{FLM83,Lyubich82,Lyu83entropy}, and by $L(\lambda) = \int \log |f'_\lambda|\mu_\lambda$ the Lyapunov exponent of $\mu_\lambda$. The following lemma collects the properties of the function $L(\lambda)$ that we shall need in the sequel.  Observe
that, as $\log |f'_\lambda|$ is not necessarily continuous on the support of $\mu_\lambda$, the first assertion is not a direct
consequence of the equidistribution of the periodic 
points with respect to $\mu_\lambda$ \cite{Lyubich82}, but requires a more quantitive equidistribution, see 
for instance \cite{BDM,BD19}.
The second
assertion is the classical Przytycki formula
\cite{Przytycki85} for polynomials.

\begin{lem}\label{lem:lyap}
For every $\lambda \in {\rm Rat}^{cm}_D$, we have
\begin{itemize}
\item[{\it (1)}] 
$L(\lambda) =\lim_{n\to \infty} D^{-n}
\sum_{x \in {\rm Pern}(\lambda)} \log |f'_\lambda (x)|$,
\end{itemize}
where ${\rm Pern}(\lambda)$ is the set of 
periodic points of $f_\lambda$ of
(exact or not exact) 
period $n$. 
Moreover, for every $\lambda \in {\rm Poly}^{cm}_D$, we 
also have
\begin{itemize}
\item[{\it (2)}] 
$L(\lambda)= \log D + \sum_{j=1}^{D-1} G_\lambda (c_j (\lambda))$,
\end{itemize}
where $G_\lambda(z) \defeq \lim_{n \to \infty} D^{-n}
\log \max\{1, |f^n_\lambda(z)|\}$ is the
\emph{Green function}
of $f_\lambda$.
\end{lem}

Let now $\Lambda_0$ be an algebraic subfamily of ${\rm Rat}^{cm}_D$ (resp.\ ${\rm Poly}^{cm}_D$).

\begin{defn}
We say that $\Lambda_0$ is given by
\emph{critical relations}  if it is given by a finite number of non-trivial
equations of
the form
\[
f^{n_1}_\lambda 
(c_{i_1}(\lambda)) 
= f^{n_2}_\lambda (c_{i_2}(\lambda)).
\]
We say that $\Lambda$
is a 
\emph{Misiurewicz family}
if it is the open subset of a 
family given by a finite number of 
critical relations of the form
\[
f^{n_i} (c_{i} (\lambda)) = f^{n_i+m_i} (c_{i}(\lambda)) = r_{i}(\lambda),
\]
where each $r_{i}(\lambda)$ is a repelling periodic point for every $f_\lambda\in \Lambda$
and each $m_i$ is strictly positive.
\end{defn}

We refer to 
\cite{BB09,BB11,demarco01,Dujardin14,DujardinFavre08,GOV19,Okuyama14}
for the description and
distribution in the parameter space  of these submanifolds 
and 
to
\cite{Berteloot13, BBD18, B19, Dujardin14, Levin82}
for their role in the
understanding of the bifurcation phenomena and related problems.

\medskip

 It is clear that the quotient by 
  M\"obius or
 affine conjugacies is well-defined on Misiurewicz families, and we will use the same name for the images of these families by the quotient map.
By a slight abuse of notation, we can think of Misiurewicz subfamilies of ${\rm Poly}_D^{cm}$ (resp.\ ${\rm poly}_D^{cm}$) also as Misiurewicz subfamilies of ${\rm Rat}_D^{cm}$ (resp.\ ${\rm rat}_D^{cm}$), where we allow (and require) the further critical relation
$f_\lambda^{-1} (\{\infty\})=\{\infty\}$.

\medskip

Let now $\Lambda$ be a Misiurewicz family and $\Lambda_0$ the algebraic family given by the associated critical relations.
Recall that a critical point $c$ is called {\it passive} on
an open subset
$\Lambda'$ of $\Lambda$ (or $\Lambda_0$)
if the sequence of holomorphic functions
$\{\lambda \mapsto f_\lambda^n(c(\lambda))\}_{n\ge 1}$ is a normal family on $\Lambda'$. Otherwise, the critical point is called {\it active} on $\Lambda'$.
It follows from the definition 
that
the critical points $c_i$
which are preperiodic to the periodic points
$r_i$
are passive on $\Lambda$ (this is true also on $\Lambda_0$, although it may be needed to pass to a finite cover to be able to follow the periodic points $r_i$ on all $\Lambda_0$).

\medskip

Recall 
\cite{Lyu83typical,MSS83} that an open subset $\Omega$ of $\Lambda$ or
$\Lambda_0$ 
is in the \emph{stability locus} of the family
if all critical points are passive on $\Omega$. We also say that the family is stable on $\Omega$.
We say that
$\Omega$ is a \emph{stable component} if it is a connected component of the stability locus.

\begin{defn}
We say that a stable component $\Omega\Subset \Lambda$ is {\it $\Lambda$-hyperbolic} if, for every $\lambda \in \Omega$, every critical point $c_j(\lambda)$ such that
$c_j$ is active on $\Lambda$ is contained in the basin of some attracting cycle for $f_\lambda$.
\end{defn}

We say that $\Omega$ is a \emph{bounded}
$\Lambda$-hyperbolic component if we have $\Omega \Subset \Lambda$ for the topology induced by $\Lambda$.
In this case, if
$\Lambda \subseteq {\rm Poly}_D^{cm}$,
all the critical points which are active on $\Lambda$ are contained in the basin of some attracting cycle 
in $\mathbb C$
(rather than just $\mathbb P^1(\mathbb C)$)
for every $\lambda\in \Omega$.
Observe that
in this case, as all critical points have bounded orbit,
by Lemma \ref{lem:lyap} (2)
we also have \begin{equation} \label{eq_Lyap_const}
L(\lambda) \equiv \log D \text{ on } \Omega.
\end{equation}

\section{Analyticity of pressure functions}
\label{sec_3}
The main goal of this section is to prove Theorem \ref{thm_ana_main}. As we will see, we in fact prove a stronger result (see Theorem \ref{t:new-analytic-all}) and obtain Theorem \ref{thm_ana_main} as a corollary (see Corollary \ref{c:delta-ana}). To prove Theorem \ref{t:new-analytic-all}, we first generalize the spectral study
carried out by Makarov-Smirnov \cite{MakSmir03}. 
Our generalization is twofold -- we extend the study in \cite{MakSmir03} (which was done for a single Misiurewicz map) to a stable holomorphic family of Misiurewicz maps, and we consider spectral properties for a more general 4-parameter family of transfer operators $\mathcal{L}_{(t_1,\lambda_1,t_2,\lambda_2)}$ which encode both the geometric potential at a given parameter and its perturbation in the direction of 
another parameter.
Then we combine it with a general method which allows us to deduce analyticity of the pressure function; see for instance \cite{SU10}. 

\medskip

Let $\Omega$ be a $\Lambda$-hyperbolic component of a Misiurewicz family $\Lambda$. We assume
for simplicity 
that the family is defined by a single critical relation $f_\lambda (c(\lambda))=r(\lambda)$, as the arguments are essentially the same in the general case.

Fix $\lambda_0 \in \Omega$ and consider a neighborhood $N(\lambda_0)\Subset \Omega$ of $\lambda_0$.
For each $\lambda \in N(\lambda_0)$, we denote by $f_\lambda$ a representative of the conjugacy class $\lambda$
and by
$J_\lambda$ the Julia set of $f_\lambda$.
Recall that, by \cite{MSS83, Lyu83typical}, there exists a \emph{holomorphic motion} of the sets
$J_\lambda$ over $N(\lambda_0)$, i.e.,
a family of conjugations
$h_\lambda \colon J_{ \lambda_0} \to J_\lambda$ which depend
holomorphically on $\lambda$. The family of the maps $(h_\lambda)_{\lambda \in \overline{N(\lambda_0)}}$
is compact in the space of 
$\gamma$-H\"older continuous
maps 
from $J_{ \lambda_0}$
to $\mathbb P^1$, for some $\gamma$ depending on $N(\lambda_0)$. We  can take 
 $\gamma\to 1$ 
 as
 $N(\lambda_0)$ shrinks
 to $\lambda_0$. In particular, we will always assume in what follows that
 $N(\lambda_0)$ is sufficiently small so that the associated $\gamma$ satisfies $\gamma>\gamma_0$ for some $\gamma_0<1$ whose precise value will be chosen when needed.
  Recall also that all these conjugations can be extended to 
 quasiconformal homeomorphisms
 from
$\mathbb P^1(\mathbb C)$ to itself,
and so
 in particular to H\"older continuous maps.

\medskip

For every $t \in \mathbb R$ and $\lambda \in \Omega$,
the {\it pressure function} is defined as
\begin{equation}\label{def_pressure}
p (t,\lambda)
\defeq
\limsup_{n\to \infty}
\frac{1}{n} \log \sum_{z\in f_\lambda^{-n} (c(\lambda))} 
|
(f_\lambda^n)' (z)|^{-t}.  
\end{equation}
We refer to \cite{PRLS04} for a number of equivalent definitions and characterizations of the pressure. 
For any fixed $\lambda$, the function
$t\mapsto p(t,\lambda)$ is strictly decreasing.
By \cite{McMullen00}, for every $\lambda\in \Omega$, the unique 
zero of the map
$t\mapsto p(t,\lambda)$
is equal to the Hausdorff dimension of the Julia set of $f_\lambda$. We denote this number by $\delta(\lambda)$.

\medskip

 In the following, it will be useful to observe that, thanks to the conjugating maps $h_\lambda$, we also have
\[
p (t,\lambda)
=
\limsup_{n\to \infty}
\frac{1}{n} \log \sum_{z\in f_{\lambda_0}^{-n} (c(\lambda_0))} 
|
(f_{\lambda}^n)'
\circ h_\lambda
(z)|^{-t}
\]
for every $t\in \mathbb R$ and $\lambda \in N(\lambda_0)$.
More generally, given
$t_1,t_2\in \mathbb R$ and
$\lambda_1, \lambda_2\in\Omega$, we define a \emph{joint pressure function}
\begin{equation}\label{eq:def-pressure-gen}
P(t_1,\lambda_1, t_2, \lambda_2)
\defeq
\limsup_{n\to \infty}
\frac{1}{n} \log \sum_{z\in f_{\lambda_0}^{-n} (c(\lambda_0))} 
|(f_{\lambda_1}^n)' ({ h_{\lambda_1}}(z))|^{-t_1}
|(f_{\lambda_2}^n)' ({ h_{\lambda_2}}(z))|^{-t_2}. 
\end{equation}

Observe that we have $P(t_1,\lambda_1,0, \lambda_2)= p(t_1, \lambda_1)$
for every $t_1, \lambda_1,\lambda_2$.

\medskip

The following is the main result of this section.
We denote by $B(\lambda,R)$ the ball of radius $R$ centered at $\lambda$.
\begin{thm}\label{t:new-analytic-all}
For every $\lambda\in \Omega$ there exists $R>0$ such that the function $P(t_1,\lambda_1, t_2, \lambda_2)$ is real analytic on 
$(\delta(\lambda)-R, \delta(\lambda)+R)
\times 
B(\lambda,R)
\times (-R,
R)\times
B(\lambda,R)$.
\end{thm}

The proof of Theorem \ref{t:new-analytic-all} consists of two steps. The first step is to prove spectral properties of a 4-parameter family of transfer operators on a tower dynamical system; see Proposition \ref{p:MS-gen}. Sections \ref{sec_3.1} and \ref{sec_3.2_tower} contain the setup and preliminary estimates for the tower dynamics, adapted to a holomorphic family of Misurewicz maps. The study of transfer operators is carried out in Section \ref{sec_3.3_to}. 
The second step is to extend the potential function
 to complex parameters,
so that the 4-parameter family of transfer operators
 can be seen as
a {\it holomorphic} family of transfer operators.
This step,
carried out in Sections \ref{sec_3.4} and \ref{ss:extension-analytic}, uses a general framework of Sumi-Urbanski \cite{SU10} and allows us to apply the Kato-Rellich perturbation theory.

\begin{cor}\label{c:delta-ana}
For every $\lambda\in \Omega$ there exists $R>0$ such that the function $p(t, \lambda)$ is real analytic on 
$(\delta(\lambda)-R, \delta(\lambda)+R)
\times 
B(\lambda,R)$. In particular, the function $\lambda\mapsto \delta(\lambda)$ is analytic on $\Omega$.
\end{cor}

\begin{proof}
The first assertion is a consequence of Theorem \ref{t:new-analytic-all}, applied with $t_2=0$. The second assertion follows from the first one and the implicit function theorem.
\end{proof}

\begin{rmk}
The potential functions appearing in \eqref{eq:def-pressure-gen}
are in fact the same 
used to define and 
study the so-called {\it Manhattan curves}
\cite{Burger93,Sharp98}. 
In \cite{BH25}, among other things, we 
study the Manhattan curves associated to two hyperbolic rational maps, and more generally
holomorphic endomorphisms of $\mathbb C\mathbb P^k$. Theorem \ref{t:new-analytic-all} in particular also shows that the Manhattan curve associated to two maps in the same $\Lambda$-hyperbolic component 
 of a Misiurewicz family
is analytic.
\end{rmk}

\subsection{Branched covers and generalized conformal Cantor sets} \label{sec_3.1}
We first recall the general
setup of \cite{MakSmir03},
which describes a model dynamics $F \colon P_1 \to U_0$ for every map in $\Omega$.
As in \cite{MakSmir03}, for simplicity, we will only consider the case of \emph{generalized conformal Cantor sets},
which corresponds to the case where all but one critical points of $F$ escape to infinity. 
However, the results and the arguments 
hold without this assumption. 
To further simplify things, we will only consider the case of our interest, i.e., we will
assume 
that the non-escaping critical point $c$ is preperiodic to a repelling cycle (while $c$ is only assumed to be non-recurrent in \cite{MakSmir03}).
Without loss of generality, up to taking an iterate of the map, we assume that the image of $c$ is a fixed point $r$,
whose multiplier has modulus equal to $\chi>1$.

\medskip

Let $U_0$ be an open Jordan domain in $\mathbb C$ 
and $P_1$ a collection of a finite number of open topological disks whose closures are disjoint and contained in $U_0$. Let $F \colon P_1 \to U_0$ be a proper analytic function.
We assume that $F$ is a branched cover of order 2 on a component $U_1$ of $P_1$ and a biholomorphism on every other component. We assume that the number of components of $P_1$ is equal to $D-1\geq 1$, so that $F\colon P_1\to U_0$ is a branched cover of degree $D$.

\medskip

Observe that the (unique) critical point $c$ of $F$ belongs to $U_1$. For every $n \ge 2$, we set
$P_n \defeq f^{-n} (U_0)$. The 
\emph{Julia set}
$J(F)$ is then given by the intersection
$\cap_n P_n$, and by assumption
we have $c\in J(F)$.
For every $n \ge 0$ and $x\in J(F)$,
we denote by
$P_n (x)$ the component of $P_n$ containing $x$. For simplicity, 
we will also set $U_n\defeq P_n (c)$.
By assumption, we have $F(c)= r\notin U_1$, and this point is a fixed repelling point for $F$.
Hence, we can also denote
$V_n \defeq P_n(r)\neq U_n$. 
Observe that $V_{n+1}$ is mapped univalently by $F$
to $V_n$ and
that $F\colon U_{n+1}\to V_n$
is a branched cover of degree $2$. 
For every $n \ge 0$, 
we also
define the {\it critical annulus} $D_n \defeq U_n \setminus \overline{U_{n+1}}$. Hence, 
$\{D_n\}_{n\ge0}$ is a collection of annuli around the critical point $c$.
For every $n \ge 2$, 
$F^n$ is a 
cover of degree $2$ from $D_n$ to $U_0\setminus V_n$. In particular,
there exist
$2(D-2)$ components
of $P_n$ 
in $D_n$, and, for each of the $D-2$ components of $P_1$ different from $U_0$,
two of them are mapped univalently
to it.

\begin{example}\label{ex:deg3}
Assume that the degree of $F$ is 
equal to $3$. In this case, $P_1$ is given by $U_1\cup V_1$,
where
$U_1 = P_1 (c)$
and
$V_1=P_1(r)$. The set $P_2$ is given by the following five sets: 
\begin{itemize}
    \item  $U_2=P_2(c)$, which is mapped by $F$ to $V_1$ with degree $2$, 
    \item $V_2 \defeq P_2 (r)$,
    which is mapped univalently by $F$ to $V_1$;
    \item a topological disc $W_2\subset V_1\setminus V_2$, which is mapped univalently by $F$ to $U_1$, and
    \item two topological disks $Y_{2,1}, Y_{2,2}
    \subset U_1\setminus U_2$, each mapped univalently
    by $F$ to $U_1$.
\end{itemize}
More generally,
for every $n\geq 2$, there exist
two components $Y_{n,1}, Y_{n,2}\subset D_{n-1}$
and a component $W_n\subset V_{n-1}\setminus V_{n}$ of $P_n$
which are all
mapped 
univalently
by $F^n$ to $U_1$. 
\end{example}

Consider now the family $(f_\lambda)_{\lambda\in N(\lambda_0)}$.
The
map $F$ 
is 
a
model for the dynamics of every $f_\lambda$ 
(to see this, it is enough to 
take as $U_0$
any connected
open neighborhood of the Julia set $J(f_\lambda)$).
To 
simplify the notations,
we will assume that $\lambda_0=0$ and,
to avoid an extra conjugation in the steps below, 
we will directly
think of $F$ as the map $f_{0}$ and $U_0$ as a connected
neighborhood of $J(f_0)$. In particular, we have $f_\lambda
\circ h_\lambda
= h_\lambda \circ f_{0} = h_\lambda \circ F$ for every $z\in J(f_0)
= J(F)$ and $\lambda \in N(0)$.
Observe that 
we restrict to $N(0)$
in order to have uniform estimates for the maps $h_\lambda$ in the sequel. The critical point $c(\lambda)$
and the repelling point $r(\lambda)$ whose relation
$f_\lambda (c(\lambda))=r(\lambda)$
defines $\Omega$ satisfy $c(0)=c$ and $r(0)=r$.

\medskip

For every $\lambda \in N(0)$, 
we let
$\chi_\lambda$ be the multiplier of the repelling fixed point $r(\lambda)$, and we also set $\hat \chi \defeq \inf_{\lambda \in N(0)} |\chi_\lambda|>1$.
The following lemma is a direct consequence of \cite[Lemma 1]{MakSmir03} applied to every $f_\lambda$ with $\lambda \in N(0)$.

\begin{lem}\label{l:L1MS}
For every $k \ge 0$, $z\in D_k$, and $\lambda \in N(0)$,
we have
\[
{\rm diam}
(h_\lambda( U_k)) \sim \chi_\lambda^{-k/2}, \quad
| (f^k_\lambda)' (h_\lambda (z))| 
\sim \chi_\lambda^{k/2},
\quad
\mbox{ and  }
\quad
|
(f^k_\lambda)'' (h_\lambda (z))|
\lesssim \chi_\lambda^k,\]
where the implicit constants
are independent of $k$, $z$,
and
$\lambda$.
\end{lem}

\subsection{Tower dynamics}\label{sec_3.2_tower}
We now recall the construction of the \emph{tower extension}
for the model dynamics $F \colon P_1 \to U_0$.
For dynamical systems exhibiting some form of hyperbolicity, the tower method has been extensively used to make the dynamics uniformly expanding on some auxiliary tower space, see for instance \cite{Hofbauer80, Young98}.

We first consider the product $U_0\times \mathbb N$ and its subset
\[
\mathcal T \defeq \sqcup_{k \geq 1} \mathcal T_k,
\]
where $\mathcal T_1 \defeq (P_1, 1)$
and 
$\mathcal T_k \defeq
(U_k, k)$
for every $k\geq 2$ are the \emph{floors} of the tower. Then, consider the map $T=T_0 \colon \mathcal T\to \mathcal T \cup (U_0,1)$
given by
\begin{equation}\label{eq:def-map-T}
T\colon (z,k)
\mapsto
\begin{cases}
(z,k+1) & z\in U_{k+1}\\
(F^k (z), 1) & z\notin U_{k+1}.
\end{cases}
\end{equation}

\begin{example}\label{ex:deg3-preim}
Consider again the case where $d=3$ as in Example \ref{ex:deg3}. We see that every point not on the first floor of $\mathcal T$ 
has exactly one preimage under $T$ 
(given by the corresponding point on the floor just below it). On the other hand, on the first floor:
\begin{itemize}
\item every point in $V_1$ has exactly one preimage in $V_2$;
\item every point in $U_1$ has infinitely many preimages: one in $W_2$, and then, for every $n\geq 2$, one in each $Y_{n,1}$ and $Y_{n,2}$.
\end{itemize}

\begin{rmk}\label{rmk_bdryUk}
Observe that all the points of the form $(z,k)$ with $z$ on the 
boundary
of $U_k$ are mapped
(to the first floor and) to the boundary of $V_1$. In particular, their images are not in $\mathcal T$.
\end{rmk}

\end{example}

Fix now a 
constant $1<\chi_* <\sqrt{\hat \chi}$.
Consider the following metric on $\mathcal T$: by definition, points in different floors have infinite distance, and the distance inside $\mathcal T_k$ is the usual Euclidean distance multiplied by the factor $(\chi_*)^{k-1}$. We will denote by $d_{\mathcal T}$ 
the induced distance function on $\mathcal T$.
Observe, in particular, that the map $F$ is expanding with respect to this distance. More precisely,
for every
$(z,k)\in \mathcal T$,
by Lemma \ref{l:L1MS}
we have
\[
T' (z,k)= \begin{cases}
\chi_* & z \in U_{k+1}\\
\chi_*^{1-k} (F^k)'(z) \gtrsim 
(\sqrt{\hat \chi } / \chi_*)^{k}
& z\notin \overline{U_{k+1}}.
\end{cases}
\]
Observe that $T'$ is
differentiable at $(z,k)$, unless $z$ belongs to the boundary of $U_{k}$.
Observe also that the diameter of $\mathcal T_k$ goes to zero exponentially by the choice of $\chi_*<\sqrt{\hat \chi}$ and Lemma \ref{l:L1MS}.

\medskip

For every $\lambda\in N(0)$,
we can similarly construct a tower dynamics $T_\lambda \colon \mathcal T_\lambda \to \mathcal T_\lambda 
\cup (h_{ \lambda} (U_0), 1)$, where
$\mathcal T_\lambda$ is the tower defined 
replacing
$P_1$ and $U_k$ with
$h_\lambda (P_1)$ and $h_\lambda (U_k)$, respectively, and $T_\lambda$ is obtained replacing
$F^k$ with
$f_\lambda^k$
in the definition 
\eqref{eq:def-map-T}
of $T$.
By the choice of $\chi_*$, for every
$\lambda$ the map
$T_\lambda$ 
is expanding on $\mathcal T_\lambda$. The following lemma is a consequence of \cite[Lemma 2]{MakSmir03} applied at every $\lambda$. It essentially follows from Lemma \ref{l:L1MS}, as
\cite[Lemma 2]{MakSmir03} follows from \cite[Lemma 1]{MakSmir03}.

\begin{lem}\label{l:L2MS}
There is a constant $C$ independent of $n$ and $\lambda$ such that if $T_\lambda^n$
is defined and differentiable at $y$, then
\[
\frac{|(T^n_\lambda)'' (y)|}{|(T^n_\lambda)' (y)|^2} \leq C.
\]
\end{lem}

The conjugations $h_\lambda$ between 
$F$ and the $f_\lambda$'s can be lifted to the towers. More precisely, 
for every $\lambda$ we can
consider the map $H_\lambda$ given by
\[
H_\lambda ( (z,k))
\defeq (h_\lambda (z),k)
\quad \mbox{ for every }
(z,k)\in \mathcal T.
\]
As the maps $h_\lambda$ are $\gamma$-H\"older continuous, the same is true for the restriction of $H_\lambda$ to each level of the tower $\mathcal T$.
On the other hand, because of the dilation of the metric at every floor,
the modulus of continuity
of
$H_\lambda$
a priori gets worse and worse
as $k\to \infty$.
The following
lemma
shows that, nevertheless, the maps $H_\lambda$ 
are still
H\"older continuous 
on the full tower. 

\begin{lem}\label{l:Hlambda-logq}
Up to taking $N(0)$ and $\kappa$ sufficiently small, 
the family $\{H_\lambda\}_{\lambda \in\overline{ N(0)}}$
is compact
in the space 
of $\kappa$-H\"older continuous maps.
\end{lem}

\begin{proof}
Recall that we are assuming that $N(0)$ is so small that 
all the $h_\lambda$'s are (uniformly) $\gamma$-H\"older continuous, for some $\gamma >\gamma_0$, where $\gamma_0$ is close to $1$.

Fix $k\in \mathbb N$ and take two points $x,y\in \mathcal T_k$ with $d_{\mathcal T} (x,y)<r$
(recall that points on different floors have infinite distance). 
As mentioned above, 
by the first estimate in Lemma \ref{l:L1MS}
and the choice
$\chi_*<\sqrt{\hat\chi}$,
 we necessarily have $r\lesssim \chi_*^{k-1} 
 \hat \chi^{-k/2}\to 0$ as $k\to \infty$. 
By definition, the
points $x$ and $y$
correspond
to points $a,b\in\mathbb C$ whose Euclidean distance is less than $\chi_*^{1-k}r$.
It 
follows from the H\"older-continuity of $h_\lambda$ 
 that we have
\[
|h_\lambda (a)- h_\lambda (b)|\leq C_\gamma
(\chi_*^{1-k}r)^\gamma,
\]
where the constant $C_\gamma$ is independent of $\lambda\in N(0)$.
This implies that we have
\[
d_{\mathcal T_\lambda} (H_\lambda (x), H_\lambda (y))
\lesssim \chi_*^{k}
(\chi_*^{-k} r)^{\gamma} 
=
\chi_*^{k (1-\gamma)}
 r^{\gamma},
\]
where the
implicit constant is independent of $\lambda$ and $k$.

Taking $\gamma_0$ sufficiently close to $1$, let $\kappa >0$ be sufficiently close to $0$ such that $\chi_*< \hat \chi^{\frac{\gamma_0-\kappa}{2(1-\kappa)}}$. This is possible by the condition $\chi_* < \sqrt{\hat \chi}$.
It follows from the above bound for $r$ 
that
we have
\[
\chi_*^{k (1-\gamma)}
 r^{\gamma-\kappa} \lesssim
 \chi_*^{k (1-\gamma)}
 (\chi_*^{k} \hat{\chi}^{-k/2})^{\gamma-\kappa}
=
\chi_*^{k(1-\gamma) +k(\gamma-\kappa)} 
 \hat{\chi}^{-k(\gamma-\kappa)/2} 
 =
 \chi_*^{k(1-\kappa)}
 \hat{\chi}^{-k(\gamma-\kappa)/2}\lesssim 1,\]
 where the implicit constant is now independent of $k$.
Hence, 
we have
\[
d_{\mathcal T_\lambda} (H_\lambda (x), H_\lambda (y))
\lesssim  r^{\kappa},
\]
where again the implicit constant is independent of $k$.
This completes the proof.
\end{proof}

\subsection{Transfer operators} \label{sec_3.3_to}
We denote by $C(\mathcal T)$ the space of bounded continuous functions $g\colon \mathcal T\to \mathbb R$ endowed 
with the $L^\infty$-norm.
For every $t_1,t_2>0$ and $\lambda_1, \lambda_2 \in N(0)$,
consider the transfer operator $\mathcal L_{t_1,\lambda_1,t_2, \lambda_2} \colon C(\mathcal T) \to C(\mathcal T)$ defined as
\begin{equation*}
\mathcal L_{t_1,\lambda_1, t_2,\lambda_2} g(x)
\defeq
\sum_{T(y)=x}
g(y)
|  R_{\lambda_1} (y) |^{-t_1}
|  R_{\lambda_2} (y) |^{-t_2}
\end{equation*}
where, for 
every
$y \in \mathcal{T}_k$ and $\lambda\in N(0)$,
the value $R_\lambda(y)$ is defined as
\[
R_\lambda (y) \defeq \begin{cases}
\chi_* & y\in U_{k+1}\times \{k\}\\
\chi_*^{1-k}  \cdot 
|(f_\lambda^k)'\circ h_\lambda| & y\notin \overline{U_{k+1}}\times \{k\}.
\end{cases}
\]
Observe that, thanks to the conjugating 
maps
$H_\lambda$, 
$\mathcal L_{t_1,\lambda_1, t_2, \lambda_2}$
can be equivalently
defined
as
\[
\mathcal L_{t_1,\lambda_1, t_2, \lambda_2} g(x)
\defeq
\sum_{T (y)= x}
g(y)
|  T_{\lambda_1}' (H_{\lambda_1} (y)) |^{-t_1}
|  T_{\lambda_2}' (H_{\lambda_2} (y)) |^{-t_2}.
\]
The second writing has the advantage of allowing us to apply, for
every $\lambda_1, \lambda_2 \in N(0)$, the uniform estimates of \cite{MakSmir03} (see Lemmas \ref{l:L1MS} and \ref{l:L2MS}). 

\medskip

Observe that the operator $\mathcal L_{t_1,\lambda_1, t_2, \lambda_2}$
above is well defined
(i.e., for every $g\in C(\mathcal T)$,
the value of 
$\mathcal L_{t_1,\lambda_1, t_2, \lambda_2} g(x)$ is well defined for every $x\in \mathcal T$
and $\mathcal L_{t_1,\lambda_1, t_2, \lambda_2} g$
is continuous).
Indeed, the only points where $\mathcal L_{t_1,\lambda_1, t_2, \lambda_2} g$ 
may be not continuous (or not defined)
in  $\mathcal T \cup (U_0,1)$ are on 
$(\partial V_1,1)$, which is outside $\mathcal T$
(see Example \ref{ex:deg3-preim} for the case of degree $3$ and Remark \ref{rmk_bdryUk}).

\medskip

In \cite{MakSmir03}, the authors study the
special
case where
$t_2=0$ and 
$\lambda_1=\lambda_0$
and obtain a spectral gap for this operator on the Banach space of 
Lipschitz
functions on $\mathcal T$.
As we now allow the parameters 
$\lambda_1,\lambda_2$
to change
in $N(0)$, we will need to take into account the distortion given by the holomorphic motion $h_\lambda$ and the map $H_\lambda$
(observe that this is necessary even when $t_2=0$).
To this aim,
we will instead
work with the Banach space
$C^{\kappa}(\mathcal T)$
of $\kappa$-H\"older continuous functions endowed with the H\"older norm $\|\cdot\|_\kappa$. 
For $g \in C^{\kappa}(\mathcal T)$, we have
\[\|g\|_{\kappa} \defeq \|g\|_{\infty} +\|g\|'_{\kappa}\]
where $\|g\|'_{\kappa}\defeq \sup_{a,b\in \mathcal T} \frac{|g(a)-g(b)|}{d_{\mathcal T} (a,b)^\kappa}$
is the \emph{H\"older constant} of $g$.

\begin{lem}\label{l:cont-L}
For every $t_1, t_2\in\mathbb R$ with $t_1+t_2 >0$
and $\lambda_1, \lambda_2 \in N(0)$, the transfer operator $\mathcal L_{t_1,\lambda_1, t_2, \lambda_2}$
is a bounded linear operator on both $C(\mathcal T)$ and $C^{\kappa} (\mathcal T)$.
\end{lem}

\begin{proof}
For simplicity, we will give a proof only in the case where $d=3$
(see Examples \ref{ex:deg3} and \ref{ex:deg3-preim}),
so that the combinatorics of the problem become simpler. 
In particular, an explicit description of the preimages of points in $\mathcal T$ is given in Example \ref{ex:deg3-preim}.
The general case can be proved in exactly the same way. To simplify the notations,
we will also denote 
$\mathcal L_{t_1, \lambda_1, t_2, \lambda_2}$ by $\mathcal L$.

\medskip

As we saw above, 
$\mathcal L$
preserves $C (\mathcal T)$. 
We now show that
$\mathcal L$
is a bounded operator on $C(\mathcal{T})$. 
If $x$ is on the $(k+1)$-th floor, then 
$x$ has only one preimage $y$ under $T$, 
which is on the $k$-th floor of $\mathcal T$, and 
$y$ belongs to $U_{k+1} \times \{k\}$.
Then we have
$T'_{\lambda_1}(y) =
T'_{\lambda_2}(y) =\chi_*$ and $$|\mathcal{L}
g(x)| = |g(y)|\cdot
|R_{\lambda_1}(y)|^{-t_1} 
|R_{\lambda_2}(y)|^{-t_2} 
\le \|g\|_\infty \chi_*^{-(t_1+t_2)}.$$

Let us now take $x$ in the first floor. If $x\in V_1\times \{1\}$, then it has one preimages (in $V_2 \times \{1\}$)
and
the argument is similar to the one above.
If, instead, $x$ 
belongs to $U_1 \times \{1\}$,
then
$x$ 
has countable preimages, which 
are as described in Example \ref{ex:deg3-preim}. We denote by $y$
 the preimage in $W_2$ and,
 for every $k\geq 2$,
 by $(y_{k,1},k)$
 and $(y_{k,2},k)$
 the preimages in 
 $Y_{k,1}\times \{k\}$
 and 
 $Y_{k,2}\times \{k\}$, respectively.
We then have
\[
\begin{aligned}
|\mathcal{L}
g(x)| 
 \le & 
 \|g\|_\infty 
\Big(
|f_{\lambda_1}'(h_{\lambda_1}(y))|^{-t_1}
|f_{\lambda_2}'(h_{\lambda_2}(y))|^{-t_2}\\
& +\sum_{k\geq 2; j=1,2} 
|
\chi_*^{1-k}
(f_{\lambda_1}^k)'(h_{\lambda_1}(y_{k,j}))|^{-t_1}
|
\chi_*^{1-k}
(f_{\lambda_2}^k)'(h_{\lambda_2}(y_{k,j}))|^{-t_2}
\Big)\\
 \lesssim &
\|g\|_\infty \cdot 
\big(
\chi_{\lambda_1}^{-t_1}
\chi_{\lambda_2}^{-t_2}
+ \sum_{k}
(\chi_*/\sqrt{\chi_{\lambda_1}})^{kt_1}
(\chi_*/\sqrt{\chi_{\lambda_2}})^{kt_2}
\big)\lesssim \|g\|_\infty,
\end{aligned}\]
where in the second step
we used Lemma \ref{l:L1MS} and
in the last one 
the assumptions $\chi_* < \sqrt{\hat \chi}$ and $t_1+t_2>0$.

Therefore, we have
$\|\mathcal{L}
g\|_\infty 
\lesssim \|g\|_\infty $ 
for some implicit constant independent of $g$.

\medskip

We omit the proof of the boundedness on $C^\kappa (\mathcal T)$ as we will prove a more precise estimate in Lemma \ref{l:L4MS}.
\end{proof}

\begin{defn}
For $\mathcal{B} = C(\mathcal T)$ 
or $C^{\kappa}(\mathcal{T})$, we define
\begin{enumerate}
\item $\rho (\mathcal L_{t_1,\lambda_1, t_2, \lambda_2},
\mathcal{B})$ to be the spectral radius of $\mathcal{L}_{t_1,\lambda_1, t_2, \lambda_2} 
\colon \mathcal{B} \to \mathcal{B}$;
\item $
\rho_{ess} (\mathcal L_{t_1,\lambda_1, t_2, \lambda_2},
    \mathcal{B})
\defeq\inf \{
\rho 
(\mathcal L_{t_1,\lambda_1, t_2, \lambda_2} -K, \mathcal{B})\colon \mbox{rank } K < \infty
\};
$
\item $\eta (t_1, \lambda_1, t_2, \lambda_2)
\defeq \rho (\mathcal L_{t_1,\lambda_1, t_2,\lambda_2}, C(\mathcal T))$.
\end{enumerate}
\end{defn}

The main result of this section is the following proposition which states that, under suitable conditions, the transfer operator $\mathcal L_{t_1,\lambda_1, t_2,\lambda_2}$ is quasicompact on $C^{\kappa}(\mathcal{T})$.

\begin{prop}\label{p:MS-gen}
If 
$t_1,t_2 \in \mathbb R$
with $t_1+t_2>0$ and $\lambda_1, \lambda_2 \in N(0)$ are such that 
 \begin{equation}\label{e:t}
P(t_1,\lambda_1,t_2,\lambda_2) > - 
(\log \hat \chi/2) (t_1+t_2),
\quad \mbox{ i.e., }
\quad 
e^{P(t_1,\lambda_1,t_2,\lambda_2)} > \hat \chi^{- (t_1+t_2)/2
},\end{equation}
then 
the following properties hold:
\begin{enumerate}
    \item $\eta(t_1,\lambda_1, t_2, \lambda_2) = e^{P(t_1,\lambda_1, t_2, \lambda_2)}$;
    \item $\rho_{\rm ess}(\mathcal L_{t_1,\lambda_1, t_2,\lambda_2},
    C^{\kappa}(\mathcal{T})) < \rho(\mathcal L_{t_1,\lambda_1, t_2, \lambda_2},C^{\kappa}(\mathcal{T}))=\eta(t_1,\lambda_1, t_2, \lambda_2)$;
    \item $\eta(t_1,\lambda_1, t_2, \lambda_2)$
    is a simple eigenvalue of $\mathcal L_{t_1,\lambda_1, t_2, \lambda_2}$ 
    in $C^{\kappa}(\mathcal{T})$.
\end{enumerate}
\end{prop}

The first assertion can be proved in the same way as in \cite[Section 2.4]{MakSmir03}, 
by using Lemma \ref{l:L1MS} instead of \cite[Lemma 1]{MakSmir03}. 
The key point in the proof of Proposition 
\ref{p:MS-gen} is the following lemma, giving a 
Lasota-Yorke
estimate
for the iterated transfer operator  with respect to the semi-norm $\|\cdot\|'_{\kappa}$. Once Lemma \ref{l:L4MS} 
is established, 
the rest of the proof carries on as in \cite{MakSmir03} with minimal modifications.
Assumption \eqref{e:t} is used in 
that part
of the proof. Observe that $t_1+t_2>0$ plays the role of $t$ in \cite{MakSmir03}, as in Lemma \ref{l:cont-L}.

\begin{lem}\label{l:L4MS}
For all 
$t_1,t_2 \in \mathbb R$ with $t_1+t_2>0$
and $\lambda_1,\lambda_2\in N(0)$,
there exist a sequence $c_n$ with $0 < c_n\to 0$ 
and a constant $C>0$ such that for
every $n\ge1$, and every $g \in C^{\kappa}(\mathcal T)$, 
we have
$$\|\mathcal L^n_{t_1,\lambda_1, t_2, \lambda_2}
g\|'_{\kappa} \le c_n
\eta(t_1,\lambda_1, t_2, \lambda_2)^n
\|g\|'_{\kappa} +C \|\mathcal L^n_{t_1,\lambda_1, t_2, \lambda_2}|g| \|_{\infty}.$$
\end{lem}

\begin{proof}
As above, for ease of notation, we will only give the proof for the case where $d=3$ and we will simply
denote $\mathcal L_{t_1, \lambda_1, t_2, \lambda_2}$
by $\mathcal L$.
We also only consider the restriction to $U_1 \times \{1\}$, as the estimates are simpler otherwise (since only one preimage is involved and the system is expanding).

\medskip

The $n$-th iterate of $\mathcal L$ is given by
\[
\mathcal{L}
^n g(x) = \sum_{T^n(y) = x} g(y) \cdot
\left(
\prod_{j=0}^{n-1}
|R_{\lambda_1}(T^{j}(y))|^{-t_1}
|R_{\lambda_2}(T^{j}(y))|^{-t_2}
\right)
\quad \mbox{ for all }
g \in C(\mathcal T).
\]
Hence, 
we have
\[
\begin{aligned}
\|\mathcal L^n
g\|'_{\kappa} &=
\sup_{\substack{a,b\in \mathcal T}}
\frac{1}{d_{\mathcal{T}}(a,b)^{\kappa}} \cdot|\mathcal L^n
g(a)-
\mathcal L^n
g(b)|\\
& =\sup_{\substack{a,b\in \mathcal T}}
\frac{1}{d_{\mathcal{T}}(a,b)^{\kappa}} \cdot
\Big|\sum_{T^n(y_a) = a}  g(y_a)
\Big(
\prod_{j=0}^{n-1}
|R_{\lambda_1}(T^{j}(y_a))|^{-t_1}
|R_{\lambda_2}(T^{j}(y_a))|^{-t_2}
\Big)
\\
& \quad  \quad -\sum_{T^n(y_b) = b} g(y_b)
\Big(
\prod_{j=0}^{n-1}
|R_{\lambda_1}(T^{j}(y_b))|^{-t_1}
|R_{\lambda_2}(T^{j}(y_b))|^{-t_2}
\Big)
\Big|.
\end{aligned}\]

We now recall (see Example \ref{ex:deg3-preim}) that we can pair the $n$-preimages of $a$ and $b$ under $F^n$
depending on the set of $P_n$ to which they belong. We will then denote these preimages by $\{a_i\}$ and $\{b_i\}$,
meaning that, for every $i_0$, $a_{i_0}$
and $b_{i_0}$ belong to the same component of $P_n$. The above expression is then bounded by $I+ II+III+IV$, where

\begin{align*}
I & \defeq \sup_{\substack{a,b\in \mathcal T}}
\frac{1}{d_{\mathcal{T}}(a,b)^{\kappa}} \cdot \sum_{i} |g(a_i)-g(b_i)|
\cdot
\Big(
\prod_{j=0}^{n-1}
|R_{\lambda_1}(T^{j}(a_i))|^{-t_1}
|R_{\lambda_2}(T^{j}(a_i))|^{-t_2}
\Big),\\
II & \defeq
\sup_{\substack{a,b\in \mathcal T}}
\frac{1}{d_{\mathcal{T}}(a,b)^{\kappa}} \cdot \sum_{i}|g(b_i)|\cdot  \Big|
\prod_{j=0}^{n-1}
|R_{\lambda_1}(T^{j}(b_i))|^{-t_1}
-
\prod_{j=0}^{n-1}
|R_{\lambda_1}(T^{j}(a_i))|^{-t_1}
\Big|
\cdot
\prod_{j=0}^{n-1}
|R_{\lambda_2}(T^{j}(b_i))|^{-t_2},
\\
III & \defeq
\sup_{\substack{a,b\in \mathcal T}}
\frac{1}{d_{\mathcal{T}}(a,b)^{\kappa}} \cdot \sum_{i}|g(b_i)|\cdot 
\prod_{j=0}^{n-1}
|R_{\lambda_1}(T^{j}(b_i))|^{-t_1}
\cdot
\Big|
\prod_{j=0}^{n-1}
|R_{\lambda_2}(T^{j}(b_i))|^{-t_2}
-
\prod_{j=0}^{n-1}
|R_{\lambda_2}(T^{j}(a_i))|^{-t_2}
\Big|,\quad \mbox{ and }\\
IV & \defeq
\sup_{\substack{a,b\in \mathcal T}}
\frac{1}{d_{\mathcal{T}}(a,b)^{\kappa}} \cdot \sum_{i}|g(b_i)|\cdot 
\prod_{j=0}^{n-1}
|R_{\lambda_1}(T^{j}(b_i))|^{-t_1}
-
\prod_{j=0}^{n-1}
|R_{\lambda_1}(T^{j}(a_i))|^{-t_1}
\Big|
\cdot\\
& \quad \quad
\quad \quad
\quad \quad
\quad \quad
\quad \quad
\quad \quad
\quad \quad
\quad \quad  \Big|
\prod_{j=0}^{n-1}
|R_{\lambda_2}(T^{j}(b_i))|^{-t_2}
-
\prod_{j=0}^{n-1}
|R_{\lambda_2}(T^{j}(a_i))|^{-t_2}
\Big|.
\end{align*}

For $I$, we have
\begin{align*}
I &
\le \|\mathcal{L}
^n 1\|_\infty \cdot \sup_{\substack{a,b\in \mathcal T}}
\frac{\max_{i} |g(a_i)-g(b_i)|}{d_{\mathcal{T}}(a,b)^{\kappa}}
\le \|\mathcal{L}
^n 1\|_\infty \cdot \sup_{\substack{a,b\in \mathcal T}}
\frac{\max_i \|g\|_{\kappa}' \cdot d_{\mathcal{T}}(a_i,b_i)^\kappa}{d_{\mathcal{T}}(a,b)^{\kappa}}  \\
& \le
\|\mathcal{L}
^n 1\|_\infty \cdot \|g\|_{\kappa}'\cdot c_n\lesssim
\eta(t_1,\lambda_1, t_2, \lambda_2)^n
\cdot\|g\|_{\kappa}'\cdot  c_n
\end{align*}
for some $c_n$ as in the statement.
In the third inequality, we have used the fact that $\max_i d_\mathcal{T}(a_i,b_i) \le d_{\mathcal{T}}(a,b)\cdot c'_n$ for some $c'_n>0$ 
with $c'_n \to 0$ as $n \to \infty$,
since $T$ is expanding.

\medskip

For $II$, observe that we have $II=0$ if $t_1=0$. Otherwise, we compute
\begin{align*}
II & =\sup_{\substack{a,b\in \mathcal T}}
\frac{1}{d_{\mathcal{T}}(a,b)^{\kappa}}
\cdot \sum_{i}|g(b_i)|\cdot  \left|
|(T_{\lambda_1}^n)'(H_{\lambda_1}(a_i))|^{-t_1} - |(T_{\lambda_1}^n)'(H_{\lambda_1}(b_i))|^{-t_1}
\right|
\cdot
\left|
(T^n_{\lambda_2})' (H_{\lambda_2} (b_i))
\right|^{-t_2}
\\
& \lesssim
 |t_1|
\cdot 
\sup_{\substack{a,b\in \mathcal T}} \frac{1}{d_{\mathcal{T}}(a,b)^{\kappa}} \cdot
\sum_{i}
|g(b_i)|
 \left|
\frac{
(T_{\lambda_1}^n)''(H_{\lambda_1}(b_i))}{(T_{\lambda_1}^n)'(H_{\lambda_1}(b_i))|^{t_1+2}}
 \right|
 \cdot
\left| (T_{\lambda_2}^n)'(H_{\lambda_2}(b_i))\right|^{-t_2}
 \cdot d_{\mathcal{T}}(H_\lambda(y_a),H_\lambda(y_b))
\\
 & \lesssim 
  |t_1|
  \cdot
\|\mathcal{L}^{n}
|g|\|_\infty \cdot\sup_{\substack{a,b\in \mathcal T}} \frac{1}{d_{\mathcal{T}}(a,b)^{\kappa}} \cdot \max_{i} \frac{|(T^n_\lambda)''(H_\lambda(b_i))|}{|(T^n_\lambda)'(H_\lambda(b_i))|^2}
 \cdot 
  d_{\mathcal{T}}(H_\lambda(a_i),H_\lambda(b_i))\\ 
 & \lesssim
|t_1| \cdot \|\mathcal{L}^{n}
 |g|\|_\infty \cdot\sup_{\substack{a,b\in \mathcal T}} \frac{1}{d_{\mathcal{T}}(a,b)^{\kappa}}
 \cdot
 d_{\mathcal{T}}(a,b)^\kappa
 = |t_1| \cdot \|\mathcal{L}^{n}
 |g|\|_\infty,
\end{align*}
where in the last 
line
we used 
Lemmas \ref{l:L2MS} and
\ref{l:Hlambda-logq}.
For $III$, we can use the same steps as for $II$,
up to reversing the role of $(t_1, \lambda_1)$
and $(t_2,\lambda_2)$. The term $IV$ is similarly
bounded by 
$\|\mathcal{L}^{n} |g|\|_\infty$
using the same estimates as in $II$
and Lemma \ref{l:L2MS}
on both the differences of products inside the sum (in particular, we have $IV=0$ if either $t_1=0$ or $t_2=0$).
The proof is complete. 
\end{proof}

\subsection{Holomorphicity of complex transfer operators} \label{sec_3.4}
To prove Theorem \ref{t:new-analytic-all}, we will need to establish the \emph{joint analyticity} of the pressure function $P$ in all its parameters. 
To do this, we will need to consider suitable \emph{complex} extensions of the operators $\mathcal L_{t,\lambda}$. This will be done in the next section. In this section, we collect the definitions and tools that we will need concerning this extension.
Observe that Proposition \ref{p:SU-full} will be used also in Section \ref{sec_5}
to prove the analyticity of the pressure form on the unit tangent bundle.

\medskip

Let $W$ be an open subset of $\mathbb{C}^\ell$ for some $\ell \ge 1$. We let $\mathcal L_w$ be a family of operators of the form
\[
\mathcal L_w g (x)
\defeq
\sum_{T(y)=x} e^{\zeta_w (y)}
g (y),
\]
where $\zeta_w \colon T^{-1} (\mathcal{T}) \to \mathbb{C}$ is a complex-valued function.
Assuming
that,
for every $w$,
$\mathcal L_w$ is a bounded linear
operator on the space
$C^{\kappa}(\mathcal T, \mathbb{C})$
of complex valued $\kappa$-H\"older 
continuous functions on $\mathcal T$,
we say that 
the
map $W \ni w \mapsto \mathcal{L}_w$ 
is {\it holomorphic} if for every $w \in W$, there exists
a bounded operator
$\mathcal{L}'_w$ on $C^{\kappa} (\mathcal T)$ 
such that $||
h^{-1}  ( \mathcal{L}_{ w+h}-\mathcal{L}_w )
-\mathcal{L}'_w||_{\kappa} \to 0$ as $h\to0$.

\begin{prop}\label{p:SU-full}
Let $\{\zeta_w\}_{w \in W}$
be a family of complex valued functions 
on $T^{-1} (\mathcal T)$.
Assume also
that
\begin{enumerate}
\item there exists $w_0 \in W$
such that $\mathcal L_{{w_0}}$ is a bounded operator on
$C^{\kappa}(\mathcal T, \mathbb{C})$;

\item the 
family of
functions $\{w \mapsto \sum_{\hat T} e^{\zeta_w} \circ \hat{T}\}_{w\in W}$
is uniformly
continuous with respect to $\|\cdot\|_{\kappa}$ on 
$U_1 \times \{1\}$,
where $\hat  T:
U_1 \times \{1\}
\to
T^{-1}(
U_1 \times \{1\}
)$
ranges over all
inverse branches of $T$ on 
$U_1 \times \{1\}$;
\item 
the family of functions $\{\zeta_w\}_{w\in W}$ 
is uniformly
continuous with respect to $\|\cdot\|_{\kappa}$ on $\mathcal T\setminus
(U_1 \times \{1\})$;

\item for every $x \in \mathcal T$,
the function $W \ni w \mapsto \zeta_w(x)$ is holomorphic.
\end{enumerate}
Then, the map $W \ni w \mapsto \mathcal{L}_{w}$ 
is holomorphic (and in particular consists of bounded operators). 
\end{prop}

\begin{rmk}
The third condition in Proposition \ref{p:SU-full}
could be replaced by
\begin{itemize}
\item[(3')] 
 the family of functions $\{e^{\zeta_w} \circ T^{-1}\}_{w\in W}$ 
is uniformly
continuous with respect to $\|\cdot\|_{\kappa}$ on $\mathcal T\setminus
(U_1 \times \{1\})$;
\end{itemize}
which we recognize as the same condition as (2) for points
in
$\mathcal T \setminus 
(U_1 \times \{1\})$, which have only one preimage under $T$.
\end{rmk}

\begin{proof}
As in \cite[Proposition 2.9]{BH24},
it is enough to show that the map $W \ni w \mapsto \mathcal{L}_{w}$ is continuous as a family of operators on $C^{\kappa}(\mathcal T,\mathbb{C})$. This, together with the fourth assumption, implies the holomorphicity of $\mathcal L_w$; see for instance \cite[Lemma 7.1]{UZ04real} or \cite[Lemma 5.1]{SU10}.

For every $x \in 
U_1 \times \{1\}$, we have
\begin{align*}
|(\mathcal{L}_w-\mathcal{L}_{w'})g(x)| & = 
\Big|\sum_{T(y)=x}\left(e^{\zeta_w(y)} - e^{\zeta_{w'}(y)}\right)g(y)\Big| 
 \le \|g\|_\infty \cdot \Big|\sum_{T(y)=x}(e^{\zeta_w} - e^{\zeta_{w'}})(y)\Big| \\
&= \|g\|_\infty \cdot \Big|\Big(\sum_{\hat T} e^{\zeta_w} \circ \hat T - \sum_{\hat T} e^{\zeta_{w'}}\circ \hat T \Big)  (x)\Big|.
\end{align*}

Therefore,
we have
$$\|(\mathcal{L}_w-\mathcal{L}_{w'})g\|_{L^\infty (
U_1 \times \{1\})} 
\lesssim \|g\|_\infty \cdot  \Big|
\sum_{\hat T} e^{\zeta_w} \circ \hat{T} - \sum_{\hat T} e^{\zeta_{w'}} \circ \hat{T}\Big|_{L^\infty (
U_1 \times \{1\}
)}.$$
The last term goes to zero as $w'\to w$
thanks to the second assumption.
The same steps as above can be used
to show the continuity in $L^\infty (\mathcal T\setminus 
(U_1 \times \{1\}))$, using the third assumption (where the sum over the inverse branches
contains now only one element).

\medskip

We then check the continuity 
with respect to
the semi-norm 
$\|\cdot\|'_\kappa$.
As above, we check it on 
$U_1 \times \{1\}$
using the second assumption, and the check on $\mathcal T\setminus
(U_1 \times \{1\})$ is completely analogous (and simpler) using the third assumption.

For every
$x,x' \in 
U_1 \times \{1\}$
and $w,w'\in W$,
we have
\begin{align*}
| (\mathcal{L}_w-\mathcal{L}_{w'})g(x)  -& (\mathcal{L}_w  -\mathcal{L}_{w'})g(x')| \\
& = \Big|\sum_{T(y)=x}e^{\zeta_w(y)}g(y)- e^{\zeta_{w'}(y)}g(y) - \sum_{T(y')=x'}e^{\zeta_w(y')}g(y')- e^{\zeta_{w'}(y')}g(y')\Big|\\
&\le 
\Big|\sum_{\hat T}
\big(
e^{\zeta_{w}(\hat T (x))}-e^{\zeta_{w'}(\hat T (x))}
\big)
\big(
g(\hat T (x))-g(\hat T(x'))
\big) 
\Big|\\
& \quad \quad + \|g\|_\infty \cdot 
\Big|
\sum_{\hat T}
(e^{\zeta_{w}(\hat T (x))}-e^{\zeta_{w'}(\hat T (x))})-(e^{\zeta_{w}(\hat T (x'))}-e^{\zeta_{w'}(\hat T (x'))})\Big|.
\end{align*}

Therefore, we have
$$
\begin{aligned}
\big\|
\big(
(\mathcal{L}_w-\mathcal{L}_{w'})g
\big)
|_{
U_1 \times \{1\}}
\big\|_\kappa'
 \lesssim & \Big\|
 \sum_{\hat T} e^{\zeta_{w}} \circ \hat T - \sum_{\hat T} e^{\zeta_{w'}} \circ \hat T \Big\|_{L^\infty (
  U_1 \times \{1\}
 )}
\cdot \|g\|_\kappa'\\
&+ \|g\|_\infty \cdot 
\Big\|
\big(
\sum_{\hat{T}} e^{\zeta_{w}} \circ \hat{T} - \sum_{\hat{T}} e^{\zeta_{w'}}\circ \hat{T}
\big) |_{
 U_1 \times \{1\}
}\Big\|_\kappa'.
\end{aligned}$$
As before, the last term goes to zero as $w'\to w$ thanks to the second assumption.
This completes the proof.
\end{proof}

\subsection{Extension of analytic functions}\label{ss:extension-analytic}
For every integer $d \ge 1$, consider the embedding $\iota_d \colon \mathbb{C}^d \to \mathbb{C}^{2d}$ given by 
\begin{equation}\label{eq_iota}
(x_1+iy_1,\ldots,x_d+iy_d) \mapsto (x_1,y_1,\ldots,x_d,y_d).
\end{equation}
Observe, in particular,
that $\mathbb C^d$ is embedded by $\iota_d$ in $\mathbb C^{2d}$ as the set of points of real coordinates
(which, in turn, is 
parametrized by $\mathbb C^d$ by means of $\iota_d$). 
For every $z \in \mathbb{C}^\ell$
and every $r>0$, denote by $D_\ell(z,r)$ 
the 
$\ell$-dimensional 
polydisk in $\mathbb{C}^\ell$
centered at $z$ and with radius $r$.
Observe that, with the above identification, we 
have $\iota_d (D_d (0,r))\subset D_{2d} (0,r)$ and, more generally,
$\iota_d (D_d (z,r))\subset D_{2d} (\iota_d(z),r)$
for every $z \in \mathbb C^d$. 

\begin{lem}[{\cite[Lemma 6.4]{SU10}}] \label{lem_6.3}
For every $M \ge 0$, 
$R>0$, 
 $\lambda_0 \in \bbC^d$, and 
 every complex analytic function
$\phi: D_d(\lambda_0,R) \to \bbC$
which is bounded in modulus by $M$, there exists a complex analytic function $\tilde{\phi} : D_{2d}(\iota_d (\lambda_0),R/4) \to \bbC$ that is bounded in modulus by $4^dM$ and such that the restriction of $\tilde \phi \circ \iota_d$ 
to
$D_d(\lambda_0,R/4)$ 
coincides with the real part $\Re(\phi)$ of $\phi$.
\end{lem}

\subsection{Proof of Theorem \ref{t:new-analytic-all}}
In order to prove
Theorem
\ref{t:new-analytic-all},
we will follow the general framework of \cite[Proof of Theorem A]{SU10}
and apply
Proposition \ref{p:SU-full} to a suitable family of complex
operators, that we now define.

First of all, 
observe that the map $\chi_\lambda\colon \Omega \to \mathbb C$ is holomorphic and 
takes value in the complement of  $\overline{\mathbb D}$.
Hence, $\log |\chi_\lambda|$ is harmonic on $\Omega$, and $|\chi_\lambda|$ has no maximum in $\Omega$. Since the assertion is local, it is enough to prove it on any $\Omega'\Subset N(0)$.
We fix such $\Omega'$ and also a $\lambda_0 \in N(0)$ with $|\chi_{\lambda_0}|> \sup_{\lambda \in \Omega'} |\chi_\lambda|$.
Recall that the holomorphic motion of the Julia sets is $\gamma$-H\"older
continuous on $N(0)$.
We allow ourself to shrink $N(0)$
during the proof, if the estimates are uniform in $\Omega'$ and $\lambda_0$.

For every $y = (z,k)
\in \mathcal T$
and $\lambda\in N(0)$, 
consider the map
$$\psi_y(\lambda) \defeq 
\begin{cases}
1 & z \in U_{k+1}\times\{k\}\\
\hat \psi_y (\lambda)
\defeq
\frac{(f^k_\lambda)' \circ h_\lambda(z)}{(f^k_{\lambda_0})'\circ h_{\lambda_0}(z)} 
& z \notin \overline{U_{k+1}}\times \{k\}
\end{cases}$$
and observe that we have 
$|\psi_y(\lambda)|= 
|R_\lambda (y)/R_{\lambda_0} (y)|$ for every $y \in \mathcal T$.
It follows from 
the choice of $\lambda_0$
and Lemma \ref{l:L1MS}
that,
shrinking $N(0)$
if necessary, for all $y \in \mathcal T$ 
and $\lambda\in \Omega'$, 
we have
$$|\psi_y(\lambda)-1|<1/5.$$
Working in chart, we can also assume that $\Omega'=
D_\ell (0,R_0)$ for some $R_0>0$ and $\ell \defeq {\rm dim}_{\mathbb C}\Omega$.
Then,
for every
$k\in \mathbb N$ and 
$y=(z,k) \in 
U_k\times \{k\}$,
there exists a branch of $\log \hat \psi_y$ on $\Omega'= D_\ell (0,R_0)$
sending 
$0$ to $0$ and whose modulus is bounded by $1/4$.
Applying Lemma \ref{lem_6.3} to the complex analytic function
$\log \hat \psi_y$, 
we see that the real analytic function
 $\Re\log \hat \psi_y \colon D_\ell (0,R_0) \to \mathbb{R}$ has an analytic extension $\Re \widetilde{\log \hat \psi_y} \colon D_{2\ell} (0,R') \to \mathbb{R}$ for some $R' \in (0,R_0)$
 which \emph{does not depend on $k$}
 and whose modulus is bounded by $4^{\ell-1}$. 
Recall that $D_\ell (0,R_0)$ is seen
as a subset of the points
of  
$\mathbb C^{2\ell}$
with real coordinates
by means of the immersion $\iota_\ell$ as in \eqref{eq_iota}, and that we have
$\iota_\ell (0)=0\in \mathbb C^{2\ell}$.

\medskip

For $(t,\lambda) 
\in
\mathbb{C} 
\times  D_{2\ell}(0,R')$,
define $\zeta_{(t,\lambda)}\colon
\mathcal T \to \mathbb C$ by

$$\zeta_{(t,\lambda)}(y)
\defeq
\begin{cases}
-t\log\chi_* &
z \in U_{k+1}\times\{k\}\\
-t\Re \widetilde{\log \hat \psi_y}(\lambda) 
-t\log(\chi_*^{1-k}|(f^k_{\lambda_0})' (h_{\lambda_0}(z))|)
& z \notin \overline{U_{k+1}}\times\{k\}.
\end{cases}$$

Observe
that we have
$e^{\zeta_{(t,\lambda)}}=
R_\lambda^{-t}$ 
for
every
$(t,\lambda) \in \mathbb{R} \times D_\ell (0,R')$.

\medskip

We want to apply Proposition \ref{p:SU-full} 
with $w=(t_1,\lambda_1, t_2, \lambda_2)$
and to
the family of operators 
$\mathcal L_{(t_1,\lambda_1, t_2, \lambda_2)} 
= \mathcal L_{\zeta_{(t_1,\lambda_2)}+
\zeta_{(t_2, \lambda_2)}}$.
For each $x \in \mathcal{T}$, the map $(t_1,\lambda_1, t_2, \lambda_2)
\mapsto \big(
\zeta_{(t_1,\lambda_1)}
+
\zeta_{(t_2,\lambda_2)}\big) (x)$ is holomorphic by the definition of $\zeta_{(t,\lambda)}$.
Moreover, $\mathcal{L}_{(\delta(0),0,0,0)}$ is a bounded operator on $C^\kappa(\mathcal{T},\mathbb{C})$ by Lemma \ref{l:cont-L}. 
The following lemma gives the second condition in Proposition \ref{p:SU-full}. The third condition can be verified by
a
similar 
(and simpler) computation, as only one inverse branch is involved.

\begin{lem} \label{lem_314}
Up to shrinking $R'$, the family of functions
\[D_1(\delta(0),R')
\times  D_{2\ell}(0,R')
\times
D_1(0,R')
\times  D_{2\ell}(0,R')\ni (t_1,\lambda_1,t_2,\lambda_2) \mapsto 
\sum_{\hat T}
e^{\zeta_{(t_1,\lambda_1)}+
\zeta_{(t_2, \lambda_2)}} \circ \hat{T}\]
is uniformly continuous with respect to the 
H\"older norm $\|\cdot\|_{\kappa}$ on
$U_1 \times \{1\}$.
\end{lem}

As we will only work on
$U_1 \times \{1\}$,
we will drop the dependence of the norm on this space.
Lemma \ref{lem_314} 
will follow from the following
two 
lemmas,
giving 
the continuity with respect to the  $\|\cdot\|_\infty$-norm and the $\|\cdot\|_\kappa'$-norm, respectively.
In all the statements and the proofs 
below, whenever an inverse branch $\hat T$
of $T$ is given, we will denote by $k=k(\hat T)$
the floor which contains the
image of $\hat T$,
which is equal to the 
 iteration exponent
as in the definition of $T$.

\begin{lem}\label{l:314a}
Up to shrinking $R'$,
there exists $L_1 > 0$
such that for any $\lambda_1,\lambda_2,\lambda'_1, \lambda'_2 
\in D_{2\ell}(0,R')$,
$t_1,t'_1 \in D_1(\delta(0),R')$, and 
$t_2,t'_2\in D_1(0,R')$
we have
$$\Big\|\sum_{\hat T} e^{
\zeta_{(t'_1,\lambda'_1)}
+
\zeta_{(t'_2,\lambda'_2)}
}\circ \hat{T} - \sum_{\hat T} e^{\zeta_{(t_1,\lambda_1)}
+
\zeta_{(t_2,\lambda_2)}
}\circ \hat{T}\Big\|_\infty \le L_1\|(t'_1,\lambda'_1, t'_2, \lambda'_2)
-(t_1,\lambda_1, t_2, \lambda_2)\|.$$
\end{lem}

\begin{proof}
We will prove a uniform continuity estimate first at fixed $t_1, t_2, \lambda_2$ and then at fixed $\lambda_1, t_2, \lambda_2$.
As the problem is essentially
symmetric
in $(t_1, \lambda_1)$ and $(t_2, \lambda_2)$
(apart from the ranges of $t_1$ and $t_2$,
see Remark \ref{r:t1t2} below for this issue),
this will give the desired continuity. In particular, in all this proof $t_2$ and $\lambda_2$ will be fixed.

\medskip

First, let us fix $t_1\in D_1 (\delta(0), R')$.
For every $\lambda_1, \lambda'_1\in D_{2\ell} (0,R')$, we have
\begin{align*}
&\Big| \sum_{\hat T} e^{\zeta_{(t_1,\lambda'_1)}+
\zeta_{(t_2,\lambda_2)}} \circ \hat{T} (w)-
\sum_{\hat T} e^{\zeta_{(t_1,\lambda_1)}+
\zeta_{(t_2,\lambda_2)}} \circ \hat{T} (w)\Big| \\
&\le \sum_{\hat T} 
\big(e^{\zeta_{(t_2,\lambda_2)}}\circ \hat T(w)\big)
e^{- \Re t_1 \cdot \log(\chi_*^{1-k}|(f^k_{\lambda_0})' (h_{\lambda_0}(\hat{T}(w)))|)} \cdot 
\Big|e^{-  t_1 \cdot \Re \widetilde{\log \hat \psi_{\hat{T}(w)}}(\lambda'_1)}  -e^{-  t_1 \cdot \Re \widetilde{\log \hat \psi_{\hat{T}(w)}}(\lambda_1)}  \Big|\\
& \lesssim
 \sum_{\hat T}
\big(e^{\zeta_{(t_2,\lambda_2)}}\circ \hat T(w)\big)
  e^{- \Re t_1 \cdot \log(\chi_*^{1-k}|(f^k_{\lambda_0})' (h_{\lambda_0}(\hat{T}(w)))|)}
\cdot 
\Big| \Re \widetilde{\log \hat \psi_{\hat{T}(w)}}(\lambda'_1)-  \Re \widetilde{\log \hat \psi_{\hat{T}(w)}}(\lambda_1)\Big|,
\end{align*}
where the implicit constant is independent of $t_1,\lambda_1, \lambda'_1$.
Since $|\Re \widetilde{\log \hat{\psi}_{\hat{T} (w)}}(\lambda)| \le 4^{\ell-1}$ 
\emph{for every inverse branch} $\hat T$
and every $\lambda$,
by Cauchy's formula
we have
$$|\Re \widetilde{\log \hat{\psi}_{\hat{T}
 (w)}}(\lambda_1) -\Re \widetilde{\log \hat{\psi}_{\hat{T}
 (w)}}(\lambda'_1)| \lesssim
|\lambda_1-\lambda'_1|,$$
where the implicit constant is independent of
$\lambda_1,\lambda'_1$ and of the branch $\hat T$ (and of $t_1$).
As $\Re t_1>0$, the term 
$e^{- \Re t_1 \cdot \log(\chi_*^{1-k}|(f^k_{\lambda_0})' (h_{\lambda_0}(\hat{T}(w)))|)} $ is 
exponentially small
for $k\to \infty$,
by the  choice of 
$\lambda_0$.
For similar reasons, the term 
$e^{\zeta_{(t_2,\lambda_2)}}\circ \hat T(w)$ is also uniformly bounded.
Therefore, we have
\begin{equation}\label{eq_3.123}
\Big\|\sum_{\hat T} e^{\zeta_{(t_1,\lambda'_1)}
+
\zeta_{(t_2,\lambda_2)}
} \circ \hat{T} - \sum_{\hat T} e^{\zeta_{(t_1,\lambda_1)}
+
\zeta_{(t_2,\lambda_2)}
} \circ \hat{T} \Big\|_\infty  \lesssim |\lambda_1'-\lambda_1|, 
\end{equation}
where the implicit constant is independent of $t_1,\lambda_1,\lambda'_1$
(and of $t_2, \lambda_2)$, as desired.

\medskip

We now fix $\lambda_1$ and consider the variation in $t_1$.
For any $\lambda_1 \in D_{2\ell}(0,R')$ and $t_1,t'_1 \in  D_1(\delta(0),R')$, 
we have
(in the first step we ignore the term
$e^{\zeta_{(t_2,\lambda_2)}}\circ \hat T(w)$ as it is uniformly bounded, as above)

\begin{align*}
& \Big|
\sum_{\hat T} e^{\zeta_{(t_1,\lambda_1)}
+
\zeta_{(t_2, \lambda_2)}} \circ \hat{T} (z)  - \sum_{\hat T} e^{\zeta_{(t'_1,\lambda_1)}
+
\zeta_{(t_2, \lambda_2)}}\circ \hat{T} (z)\Big| 
\\
& \lesssim
\sum_{\hat T} 
 e^{-\Re t_1  \Re \widetilde{\log \hat{\psi}_{\hat{T} 
(z)}}(\lambda_1)}
\cdot 
e^{- 
 \Re
t_1 \log(\chi_*^{1-k}|(f^k_{\lambda_0})' (h_{\lambda_0}(\hat{T} 
 (z)))|)} \\[-0.3cm]
& \quad \quad
 \quad \quad
  \quad \quad
   \quad \quad
    \quad \quad
     \quad \quad
     \cdot
\big(1- e^{
 (t_1 -t'_1)(\Re \widetilde{\log \hat{\psi}_{\hat{T} 
  (z)}}(\lambda_1)-  \log(\chi_*^{1-k}|(f^k_{\lambda_0})' (h_{\lambda_0}(\hat{T} 
  (z)))|))
}\big)\\[0.1cm]
& \lesssim
|t_1-t'_1|
\sum_{\hat T} 
e^{- \Re t_1 \log(\chi_*^{1-k}|(f^k_{\lambda_0})' (h_{\lambda_0}(\hat{T} 
 (z)))|)} 
\big|\Re \widetilde{\log \hat{\psi}_{\hat{T} 
 (z)}}(\lambda_1)-
\log(\chi_*^{1-k}|(f^k_{\lambda_0})' (h_{\lambda_0}(\hat{T} 
(z)))|)\big|,
\end{align*}
where in the last step we bounded the factor 
$e^{- \Re t_1 \Re \widetilde{\log \hat{\psi}_{\hat{T} 
(z)}}(\lambda_1)}$
as in the previous part of the proof. 
To conclude, denoting 
$a_k:=\log(\chi_*^{1-k}|(f^k_{\lambda_0})' (h_{\lambda_0}(\hat{T} 
 (z)))|)$
and observing that $a_k\to \infty$ for $k\to \infty$,
we just need to bound the expression
\[
\sum_{\hat T} 
e^{-\Re  t_1 a_k} 
\big|\Re \widetilde{\log \hat{\psi}_{\hat{T} 
 (z)}}(\lambda_1)-
a_k\big|. 
\]
As the  term
$\Re \widetilde{\log \hat{\psi}_{\hat{T} 
(z)}}(\lambda_1)$ is bounded and $t_1$ belongs to a small neighborhood of $\delta(0)$,
the expression in bounded and the desired bound 
\begin{equation}\label{eq_3.124}
\Big\|\sum_{\hat T} e^{\zeta_{(t_1,\lambda_1)}
+
\zeta_{(t_2,\lambda_2)}} \circ \hat{T} - \sum_{\hat T} e^{\zeta_{(t'_1,\lambda_1)}
+
\zeta_{(t_2,\lambda_2)}
} \circ \hat{T} \Big\|_\infty \lesssim |t_1-t'_1|,
\end{equation}
follows,
where
the implicit constant is independent of $t_1,t'_1,\lambda_1$
(and of $t_2,\lambda_2$).
The assertion now follows from \eqref{eq_3.123} and \eqref{eq_3.124}.
\end{proof}

\begin{rmk}\label{r:t1t2}
In the proof of Lemma \ref{l:314a}, when working with $t_2$ instead of $t_1$, we do not have an exponential bound for
$e^{- \Re t_2 \log(\chi_*^{1-k}|(f^k_{\lambda_0})' (h_{\lambda_0}(\hat{T} 
 (z)))|)}$
as $t_2$ may be equal to $0$.
On the other hand, the term 
$e^{- \Re t_1 \log(\chi_*^{1-k}|(f^k_{\lambda_0})' (h_{\lambda_0}(\hat{T} 
 (z)))|)}$
is still present (and exponentially small), and can be used to bound the expressions
in the same way as above.
\end{rmk}

\begin{lem}\label{l:314b}
Up to shrinking $R'$,
there exists a constant $L_2 \ge 1$ such that for any $\lambda_1,\lambda_2,\lambda'_1, \lambda'_2
\in D_{2\ell}(0,R')$,
$t_1,t'_1 \in D_1(  \delta(0),R')$,
 and 
 $t_2,t'_2 \in D_1(0,R')$,
we have
$$\Big\|\sum_{\hat T} e^{\zeta_{(t_1,\lambda_1)}
+
\zeta_{(t_2,\lambda_2)}
} \circ \hat{T}- \sum_{\hat T} e^{\zeta_{(t'_1,\lambda'_1)}+
\zeta_{(t'_2,\lambda'_2)}} \circ \hat{T}\Big\|_\kappa' \le L_2||(t'_1,\lambda'_1, t_2, \lambda'_2)
-(t_1,\lambda_1, t_2, \lambda_2)||.$$
\end{lem}

\begin{proof}
We first observe that, thanks to 
Lemma \ref{l:314a}
and the fact that
$e^{- \Re t_1 \cdot \log(\chi_*^{1-k}|(f^k_{\lambda_0})' (h_{\lambda_0}(\hat{T}(w)))|)} $ is 
exponentially small
for $k\to \infty$
and 
$e^{\zeta_{(t_2,\lambda_2)}}\circ \hat T(w)$ is  
uniformly bounded,
it will be enough to prove a \emph{uniform}
continuity with respect to the $\kappa$-H\"older seminorm $\|\cdot\|'_\kappa$ for all terms
$e^{\zeta_{(t_1,\lambda_1)}} \circ \hat{T}$. Namely, we do not need to control the constants in an exponentially small way, as the summable coefficients are provided by the 
bounds for $\|\cdot\|_\infty$ in 
Lemma \ref{l:314a}.
We will then prove, for every $\hat T$,
the $\kappa$-H\"older continuity of both 
the terms in $\zeta_{(t_1,\lambda_1)}\circ \hat T$. The desired bound then follows computing 
the variation of the function $e^{\zeta_{(t_1,\lambda_1)}\circ \hat T}$, which we can naturally write as the product of two exponentials, one of which 
is uniformly bounded and the other
exponentially small in $k$.

\medskip

As in Lemma \ref{l:314a}
we first consider the variation in $t_1$ at a fixed $\lambda_1$. 
For 
every
inverse branch $\hat T$ of $T$ on
$U_1\times \{1\}$
and
every $z,w\in
U_1\times \{1\}$
we have
\begin{equation}\label{eq_holder1}
\left|\log |(f_{\lambda_0}^k)'(h_{\lambda_0}(\hat{T}
 (w)))|
- \log |(f_{\lambda_0}^k)'(h_{\lambda_0}(\hat{T}
 (z)))|
\right|
\lesssim 
d(z,w)^\gamma,
\end{equation}
where the implicit constant is independent of 
$z,w,\hat T$ and
we recall that $\gamma >\kappa$ is the 
H\"older exponent of 
the holomorphic motion over $N(0)$.
Indeed,
observe that both $\hat T ( z)$ and $\hat T (w)$ belong to
$D_k\times \{k\}$.
For every
$\lambda_3\in N(0)$,
by 
Lemma 
\ref{l:L2MS}
we then have
\[
\left|
\log (f_{\lambda_3}^k)'(h_{\lambda_1}(\hat{T}
(w))) - \log (f_{\lambda_3}^k)'(h_{\lambda_1}(\hat{T}
(z)))\right|
\lesssim \frac{|(f_{\lambda_3}^k)''|}{|(f_{\lambda_3}^k)'|^ 2} C_\gamma
d(z, w)^\gamma
\lesssim
d(z, w)^\gamma,
\]
which gives 
\eqref{eq_holder1}
taking $\lambda_3=\lambda_1$.
Applying the above inequality with $\lambda_3=\lambda_1$ and $\lambda_3=\lambda_0$ and taking the difference, we obtain
\[
|\log \hat{\psi}_{\hat{T}(w)}
(\lambda_1) - 
\log \hat{\psi}_{\hat{T}(z)}
(\lambda_1)| \lesssim 
d(z,w)^\gamma.\]
We deduce from
Lemma \ref{lem_6.3} that we have
\begin{equation}\label{eq_holder2}
|\Re \widetilde{\log \hat{\psi}_{\hat{T}(w)}}(\lambda_1) - \Re \widetilde{\log \hat{\psi}_{\hat{T}(z)}}(\lambda_1)| \lesssim 
4^\ell d(z,w)^\gamma,
\end{equation}
where again the implicit constants 
are independent of
$\hat T,z,w$.
Together,
\eqref{eq_holder1} and \eqref{eq_holder2} give the H\"older continuity in $t_1$
of the two terms in the definition of $\zeta_{(t_1,\lambda_1)} \circ \hat{T}$, as desired.

\medskip

We then consider the variation in
$\lambda_1$ for a fixed $t_1$.
Observe that the second term in the definition of $\zeta_{(t_1,\lambda_1)}$ does not depend on $\lambda_1$, hence we only have to prove the continuity in $\lambda_1$ of its first term, i.e., of
$\Re \widetilde{\log \hat{\psi}_{\hat{T}(z)}}(\lambda_1)$.

Since by \eqref{eq_holder2} 
we  
have
$|\Re \widetilde{\log \hat{\psi}_{\hat{T}
(w)}}(\lambda_1) - \Re \widetilde{\log \hat{\psi}_{\hat{T}
(z)}}(\lambda_1)| \le 4^\ell\cdot 
d(z,w)^\kappa$, 
up to shrinking $R'$,
by Cauchy's formula, 
we have
$$|\Re \widetilde{\log \hat{\psi}_{\hat{T}
(w)}}(\lambda_1) - \Re \widetilde{\log \hat{\psi}_{\hat{T}
(z)}}(\lambda_1)-\Re \widetilde{\log \hat{\psi}_{\hat{T}
(w)}}(\lambda'_1) + \Re \widetilde{\log \hat{\psi}_{\hat{T}
(z)}}(\lambda'_1)| \lesssim
|\lambda_1-\lambda'_1|
\cdot
d(z,w)^\kappa
$$
for any $\lambda_1,\lambda'_1 \in D_{2\ell}(0,R')$. This gives the desired continuity and completes the proof.
\end{proof}

Lemma \ref{lem_314} is then a
consequence of Lemmas \ref{l:314a} and \ref{l:314b}.
We can now  complete the proof of
Theorem 
\ref{t:new-analytic-all}.

\begin{proof}[Proof of 
Theorem
\ref{t:new-analytic-all}]
By the above, and in particular by 
Lemma \ref{lem_314}, up to shrinking $R'$
the family of operators 
$\mathcal L_{(t_1,\lambda_1, t_2, \lambda_2)} 
= \mathcal L_{\zeta_{(t_1,\lambda_1, t_2, \lambda_2)}}$
verifies
all 
the assumptions of Proposition \ref{p:SU-full}. Therefore, by  that proposition,
the map 
$D_\ell (\delta (0) , R') \times  D_{2\ell} (0,R')
\times D_1 (0,R')\times D_{2\ell} (0,R')
\ni 
(t_1,\lambda_1, t_1, \lambda_2)
\mapsto \mathcal{L}_{(t_1,\lambda_1, t_2, \lambda_2)}\in 
C^{\kappa} (\mathcal T)$
is holomorphic. 
Thanks to Proposition \ref{p:MS-gen}, the assertion is now 
a consequence of Kato-Rellich perturbation theorem.
\end{proof}

\section{The Hessian form $\langle \cdot, \cdot \rangle_G$ and the pressure form $\langle \cdot, \cdot \rangle_{\mathcal{P}}$} \label{sec_4}
Let $\Omega$ be a $\Lambda$-hyperbolic component of a Misurewicz family $\Lambda \subset {\rm rat}_D^{\rm cm}$.
In this section, we construct the Hessian form $\langle \cdot, \cdot \rangle_G$ and the pressure form $\langle \cdot, \cdot \rangle_{\mathcal{P}}$. Thanks to
Corollary \ref{c:delta-ana}, the construction of the Hessian form $\langle \cdot, \cdot \rangle_G$ and the pressure form $\langle \cdot, \cdot \rangle_{\mathcal{P}}$ can be obtained using the same arguments as in \cite{HeNie23,BH24}. Here we give their definitions and the main steps of their constructions for completeness.
For simplicity, we will work with a lift of $\Omega$ to ${\rm Rat}_D^{\rm cm}$ and construct a metric there. It follows from standard arguments that this metric indeed descends to the quotient (and is degenerate on vectors tangent to the fibers); see \cite{HeNie23}.

\medskip

Fix $\lambda_0 \in \Omega$. Let $U(\lambda_0)$ be a neighborhood of $\lambda_0$
such that the conjugacy $h_\lambda \colon J_{\lambda_0} \to J_\lambda$
is well-defined for any $f_\lambda \in U(\lambda_0)$. 
We denote by 
$\nu= \nu_{\lambda_0}$
the unique
equilibrium state on $J_{\lambda_0}$ for the potential
$-\delta
(\lambda_0)\log |f_{\lambda_0}'|$. We note that the existence and uniqueness of $\nu$ is guaranteed by 
\cite{MakSmir03}, see also
Proposition \ref{p:MS-gen} and 
\cite{PRLS03}.

Consider the function 
$\mathrm{Ly}_{\lambda_0} \colon U(\lambda_0)\to \bbR$
given by
$$\mathrm{Ly}_{\lambda_0}(\lambda) \defeq  \int_{J_\lambda} \log|f'_\lambda| d\left((h_\lambda)_* \nu\right)=\int_{J_{\lambda_0}} \log|f'_\lambda \circ h_\lambda|d\nu.$$
We observe that, as the critical points 
in the Julia set
are
preserved by the holomorphic motion,
for every $z\in J_{\lambda_0}$
the holomorphic
function
$\lambda \mapsto f'_\lambda \circ h_\lambda (z)$
is either never zero or constantly equal to zero. Hence,
the same arguments as 
those in the proof of 
\cite[Lemma 4.2]{BH24} give that 
 $\mathrm{Ly}_{\lambda_0}$
is harmonic, hence 
in particular
real-analytic.

As a consequence of the analyticity of both $\mathrm{Ly}_{\lambda_0}$
and $\delta$ (by 
Corollary \ref{c:delta-ana}),
we see that also the function
 $G_{\lambda_0} \colon U(\lambda_0) \to \bbR$ defined as
$$G_{\lambda_0}(\lambda) \defeq \delta(\lambda)\mathrm{Ly}_{\lambda_0}(\lambda) = \delta(\lambda) \int_{J_{\lambda_0}} \log|f'_\lambda \circ h_\lambda|d\nu$$
is 
real-analytic. 
The same proof as
that of  \cite[Proposition 4.5]{BH24} gives that
$G_{\lambda_0}$ has a minimum at $\lambda_0$.
Therefore we have $DG_{\lambda_0}(\lambda_0)=0$ and the Hessian $G_{\lambda_0}''(\lambda_0) \colon T_{\lambda_0}\Omega \times T_{\lambda_0}\Omega \to \mathbb R$ is well-defined at $\lambda_0$
and it gives
a positive semi-definite symmetric bilinear form $\langle \cdot, \cdot \rangle_G$ on the tangent space $T_{\lambda_0}\Omega$. Namely,
we can set
$$\langle \vec{u},\vec{v} \rangle_G \defeq (G_{\lambda_0}''(\lambda_0))(\vec{u},\vec{v})
\quad
\mbox{ for every}
\quad
\vec{u},\vec{v} \in T_{\lambda_0}\Omega.$$
For every $\vec{v}
\in T_{\lambda_0}\Omega$,
we will also denote
$\|\vec{v}\|_G \defeq \sqrt{\langle \vec{v}, \vec{v}\rangle_G}$.
A direct computation shows that
if 
 $\gamma(t), t \in (-1, 1)$
 is a smooth
path in $U(0)$ with $\gamma(0) = \lambda_0$ and $\gamma'(0)=\vec{v} \in T_{\lambda_0}\Omega$, we have 
$$||\vec{w}||_G^2 = \left.\frac{d^2}{dt^2} \right|_{t=0} G_{\lambda_0}(\gamma(t)).$$
This characterization is used to prove that the metric, although defined in 
${\rm Rat}_D^{\rm cm}$,
indeed descends to the quotient in ${\rm rat}_D^{\rm cm}$; see for instance \cite[Section 4.2]{BH24}.

\bigskip

Now we recall the construction of the pressure form $\langle \cdot, \cdot \rangle_{\mathcal{P}}$. Recall that  $\delta(\lambda)$ is the unique real number 
satisfying $p(\delta(\lambda),\lambda) = 0.$ Recall that the pressure function $\mathcal{P}(\phi)$ is defined as $\mathcal{P}(\phi) \defeq \sup_{\mu} h_\mu(f_{\lambda_0}) + \int_{J_0}\phi d\mu$ where the supremum is taken over all the $f_{\lambda_0}$-invariant probability measures on $J_0$. By \cite{PRLS04}, the definition of $\mathcal{P}(-t\log|f_\lambda'|)$ coincides with $p(t,\lambda)$ as in \eqref{def_pressure}.

Let $\mathcal{C}
(J_{\lambda_0})$ be the set of cohomology classes of H\"older continuous functions with pressure $0$
with respect to $f_{\lambda_0}$,
that is,
$$\mathcal{C}
(J_{\lambda_0}) \defeq \{\phi: \phi \in C^{\alpha}(J_{\lambda_0},\mathbb{R}) \text{ for some } \alpha>0, \mathcal{P}(\phi) = 
0\}/ \sim$$
where $\phi_1 \sim \phi_2$ if $\phi_1$ and  $\phi_2$ are $C^0$-cohomologous
on $J_{\lambda_0}$, i.e., there exists a continuous function $g \colon J_{\lambda_0} \to \mathbb R$ such that $\phi_1 - \phi_2 = g - g \circ f_{\lambda_0}$.
For every $\lambda \in \Omega$, we denote by $\mathscr{E}(\lambda)$ the  class
$[-\delta(\lambda)\log|f'_\lambda \circ h_\lambda|]\in \mathcal C (J_{\lambda_0})$
and by
 $\nu_\lambda$
the unique 
equilibrium state for
any representative $\phi$ of 
$\mathscr E (\lambda)$, whose existence and uniqueness is guaranteed by Proposition \ref{p:MS-gen}. 
It is a standard fact 
(see for instance \cite[p. 375]{McMullen08})
that
the tangent space 
of $\mathcal{C}
(J_{\lambda_0})$ at $\mathscr{E}(\lambda)$ can be identified with 
$$T_{\mathscr{E}(\lambda)}\mathcal{C}(J_{\lambda_0}) = \left\{\psi 
\colon \psi
\in C^{\alpha}(J_{\lambda_0},\mathbb{R}) \text{ for some } \alpha>0,
\int_{J_{\lambda_0}} \psi d\nu_\lambda = 0  \right\} / \sim,$$
where we used the fact
that, by definition, the pressure is constant on $\mathcal C(J_{\lambda_0})$.
Following \cite[p. 375]{McMullen08}, we define the {\it pressure form $||\cdot||_{pm}$ on $\mathscr{E}(\Omega) \subset \mathcal{C}(J_{\lambda_0})$}
as
$$||[\psi]||^2_{pm} \defeq \frac{{\rm Var}(\psi,\nu_\lambda)}{\int_{J_{\lambda_0}} \phi d\nu_\lambda}
\quad
\mbox{ for every}
\quad
\psi \in T_{\mathscr{E}(\lambda)}\mathcal{C}(J_{\lambda_0}),$$
where 
$${\rm Var}(\psi, \nu_\lambda) \defeq \lim_{n \to \infty} \frac{1}{n} \int_{J_{\lambda_0}} \left( \sum_{i=0}^{n-1} \psi \circ f^i(x) \right)^2 d\nu_\lambda
\in [0,+\infty].$$
Finally,
we define the
\emph{pressure form $\|\cdot\|_{\mathcal P}$
on $\Omega$} 
as the pull-back
of $||\cdot||_{pm}$ by the map $\mathscr E$. Namely,
given $\vec{w} \in T_{\lambda}\Omega$, let $\gamma(t)\defeq [f_t]$,
$t \in (-1,1)$ be a smooth path in $\Omega$
with $\gamma(0)=\lambda$ and $\gamma'(0)= \vec{w}$.
Letting $\phi_t$ be a representative of the class $\mathscr{E}(\gamma(t))$,
we define
\begin{equation*}
\|\vec{w}\|_{\mathcal P}
\defeq ||\dot{\phi}_0||_{pm} =
\frac{{\rm Var}(\dot{\phi}_0,\nu_{\lambda_0})}{\int_{J_{\lambda_0}} \phi_0 d\nu_{\lambda_0}}.
\end{equation*}
Observe that 
$||\cdot||_{\mathcal{P}}$ is positive semi-definite on $T_\lambda\Omega$ since 
we have ${\rm Var}(\dot{\phi}_0,\nu_\lambda) \ge 0$
and $\int_{J_{\lambda_0}} \phi_0 d\nu_{\lambda_0} > 0$.

\bigskip

A computation as in 
\cite[Proposition 4.12]{BH24}
gives the equality
	$$||\vec{w}||_{\mathcal{P}}^2 = \frac{||\vec{w}||_G^2}{\int_{J_{\lambda_0}} \phi_0 d\nu_{\lambda_0}}.
 $$
Namely, the Hessian
form $\langle \cdot, \cdot \rangle_G$ is conformal equivalent
to the pressure form $\langle \cdot, \cdot \rangle_{\mathcal{P}}$. 
We also see that 
$||\vec{w}||_G = 0$
if and only if
$||\vec{w}||_{\mathcal{P}} =0$,
if and only if
${\rm Var}(\dot{\phi_0}, \nu) = 0$. It follows from standard facts in thermodynamical formalism that these conditions are also equivalent to the fact that 
$\dot{\phi_0}(x)$ is a $C^0$-coboundary, i.e., it is $C^0$-cohomologous to zero.
As in \cite{HeNie23},
we deduce from this last equivalence the following lemma. 

\begin{lem} \label{lem_K_equation}
If $||\vec{w}||_G=0$,
there exists a constant $K \in \bbR$ such that, for every $n\in\mathbb N$, we have 
\begin{align*}
\frac{d}{dt}\Big|_{t=0} 
S_n\big(
\log|f'_t\circ h_{{\gamma}(t)}(x)|\big)
&= K \cdot S_n
\big(
\log|f' \circ h_{{\gamma}(0)}(x)|\big)
\end{align*} 
for all 
$n$-periodic
points
$x$ of $f$ 
in $J_{\lambda_0}$.
Here $S_n\phi \defeq \phi + \phi \circ f + \cdots + \phi \circ f^{n-1}$
denotes the Birkhoff sum of $\phi$. 
\end{lem}

Given a $C^1$-path $\gamma \colon (0,1)\to \Omega$, we define the {\it length} $\ell_G(\gamma)$ of $\gamma$ as
$$\ell_G(\gamma) \defeq \int_0^1 \|\gamma'(t)\|_G dt.$$

\begin{prop}\label{p:non-deg-analytic-paths}
Let $\Omega$ be a bounded $\Lambda$-hyperbolic component of a Misiurewicz family $\Lambda$ in ${\rm poly}_D^{cm}$.
We have $\ell_G(\gamma)>0$ for any non-trivial
$C^1$-path $\gamma \colon (0,1)\to \Omega$.
\end{prop}

\begin{proof}
The proof follows the same arguments as that of \cite[Proposition 5.3]{BH24}.
We give a sketch here to show how Lemma \ref{lem_K_equation} is used and to highlight that this is the place where we use the assumption that $\Omega$ is a {\it bounded} component of a family $\Lambda\subset{\rm poly}_D^{cm}$
(which implies 
that the Lyapunov exponent of the measure of maximal entropy is constant on $\Omega$; see \eqref{eq_Lyap_const}),
as well as the need of the more refined equidistribution property for the multipliers as in Lemma \ref{lem:lyap} (1).

\medskip

Suppose by contradiction that we have
$\ell_G(\gamma)=0$. 
We denote by $\{x_i(\lambda)\}_{i\ge1}$ the set of maps parametrizing the repelling periodic points on $\Omega$, and let $n_i$ be the period
of the corresponding cycle.
Fix $s_0\in (0,1)$ and an index $i_0$.
It follows from Lemma \ref{lem_K_equation} that, for every $i\ge 1$
and every $s\in (0,1)$, we have
\begin{equation*}
S_{n_{i_0} n_i} \big(
\log |f'_{\gamma(s)} (x_i (\gamma(s)))|
\big)
= a_{i,i_0} e^{\widetilde K (s,s_0)} 
S_{n_{i_0} n_i}
\big(
\log |f'_{\gamma(s_0)} (x_{i_0} (\gamma(s_0)))|\big)
\end{equation*}
for some strictly
positive
(as both $x_i(\gamma(s))$ and $x_{i_0}(\gamma(s_0))$
are repelling)
constants $a_{i,i_0}$.

\medskip

As $\Omega$ is bounded, the Lyapunov exponent of the measure of maximal entropy is constantly
equal to $\log D$ on $\Omega$ by \eqref{eq_Lyap_const}.
Hence, it follows from Lemma \ref{lem:lyap} (1)
that,
 for every $s \in (0,1)$,
 we  have
\begin{equation} \label{eq_tildeL}
\begin{aligned}
\log D
& =
\lim_{n\to \infty}
\frac{1}{n_{i_0} n} \frac{1}{D^{n_{i_0}n}}
\sum_{x_i\colon n_i =n_{i_0}n}
 a_{i,i_0} e^{\widetilde K (s,s_0)} 
 S_{n_{i_0} n} \big(
 \log |f'_{\gamma(s_0)} (x_{i_0} (\gamma(s_0)))|\big) \nonumber\\
& = e^{\widetilde K (s,s_0)} \cdot
 \left(
 \lim_{n\to \infty}
 \frac{1}{D^{n_{i_0}n}}
 \sum_{x_i\colon n_i=n}
a_{i,i_0} 
\right)
\cdot
\frac{1}{n_{i_0}}
S_{n_{i_0}} \big(
\log |f'_{\gamma(s_0)} (x_{i_0} (\gamma(s_0)))|\big).
\end{aligned}
\end{equation}
We deduce that the function $\widetilde K(s,s_0)$ is independent of $s$.
This shows that the absolute values of all the multipliers
of the $x_i (\gamma(s))$'s are constant along $\gamma$, which contradicts the rigidity results of \cite{JiXie23}.
\end{proof}

\section{Analyticity of the metric and a distance function on $\Omega$}\label{sec_5}
The goal of this section is to prove Theorem \ref{thm_main}. Our proof follows the general framework of \cite[Section 5]{BH24}. For the case of Misiurewicz maps, the main point is 
to prove that the 2-form $\langle \cdot, \cdot \rangle_G$ is analytic on the unit tangent bundle of $\Omega$. 
For every $k \in \mathbb N$, $z \in \mathbb C^k$,
and $r>0$, we denote by $D_{k}(z,r)$ the $k$-dimensional polydisk in $\mathbb{C}^k$ centered at $z$ with radius $r$. Recall that $\ell \defeq {\rm dim}_{\mathbb C}\Omega$.

\medskip

Let $\{\vec{v}_s\}_{s\in D_1(0,R_0)}$ be a holomorphic family of elements 
of
$ \mathbb{C}^\ell \setminus\{\vec{0}\}$, 
i.e., we assume that the map $s \mapsto 
\vec{v}_s\in \mathbb C^\ell \setminus\{\vec{0}\}$ is holomorphic.
Observe also that we can
identify $T_\lambda \Omega$ with 
$\mathbb C^\ell$ 
for every $\lambda \in D_\ell (0,R_0)$.
Consider the map $\gamma\colon D_{\ell}
(0,R_0)
\times D_1(0,R_0)\times D_1 (0,R_0)\to \mathbb C^\ell$ given by
\[
\gamma (\lambda, t,s ) = \lambda + t\vec{v}_s.
\]
Up to shrinking $R_0$, we can assume that the image of $\gamma$ is contained in $\Omega$. 
Moreover, it is clear from the definition that $\gamma$ satisfies
\[
\gamma(\lambda,0, s) = \lambda 
\quad \mbox{ and }
\quad
\frac{d}{dt}\Big|_{t=0}
\gamma (\lambda, t, s) = \vec{v}_s \in T_{\lambda}{\Omega}
\quad \mbox{ for all } \lambda \in D_\ell (0,R_0)
\mbox{ and }
s \in D_1 (0,R_0).
\]

For every $(\theta,\lambda, t,s) \in \bbR \times D_\ell(0,R_0)\times D_1(0,R_0)\times D_1(0,R_0)$, we also define
$$\phi_{(\theta,\lambda, t,s)} \defeq -\delta(\lambda)\log|f'_\lambda| + \theta\log |f'_{\gamma(\lambda, t,s)}\circ\Psi_{\lambda, t,s}| \colon J_\lambda \to \bbR$$
where we denote by $\Psi_{\lambda, t,s} \colon J_\lambda \to J_{\lambda,t,s}$ the conjugacy map induced by the holomorphic motion on $\Omega$. 
Using notations as in Section \ref{sec_3}, the pressure of $\phi_{(\theta,\lambda,t,s)}$ can be written as
\[\mathcal{P}(\phi_{(\theta, \lambda, t,s)})
= P (t_1,\lambda_1,t_2,\lambda_2)
\]
where  $\lambda_1 = \lambda$, $\lambda_2 = \gamma(\lambda, t,s)$
(which is analytic in the variables), 
$t_1 = \delta (\lambda_1)$
(which is analytic in $\lambda_1$ and so in $\lambda$),
and 
$t_2=\theta$.

\begin{prop} \label{prop_analyticity_form}
There exists $0<R<R_0$ such that the map 
$(-R,R)
\times D_\ell(0,R)\times
(-R,R)
\times D_1(0,R)\ni (\theta,\lambda,t,s) \mapsto \mathcal{P}(\phi_{(\theta,\lambda,t,s)})$ is real-analytic.
In particular, the map 
$D_\ell(0,R)\times
D_1 (0,R) \ni (\lambda,s)\mapsto (G_{\lambda}''(\lambda))
(\vec{v}_s, \vec{v}_s)$ 
is real-analytic.
\end{prop}

\begin{proof}
The first assertion is a 
direct 
consequence of Theorem \ref{t:new-analytic-all}. 
The second one conclusion follows by taking derivatives of the pressure with respect to $\theta$ and then
along tangent vectors, as in the proof of \cite[Corollary 5.2]{BH24}. This completes the proof.
\end{proof}

We can now give a proof of Theorem \ref{thm_main}.

\begin{proof}[Proof of Theorem \ref{thm_main}]
By Proposition \ref{prop_analyticity_form}, the pseudo-metric $d_G$ is determined by a family of positive semi-definite bilinear forms on $T_\lambda\Omega$ depending analytically on $\lambda \in \Omega$. By the same induction argument as in \cite[Section 5.4]{BH24}, we see that any continuous path between two points $x$ and $y$ having zero length must lie in an analytic real one-dimensional submanifold of $\Omega$. If such a path existed, it would contradict Proposition \ref{p:non-deg-analytic-paths}. The assertion follows.
\end{proof}

\printbibliography

\end{document}